\documentclass[11pt]{amsart}
\usepackage{amsmath,amsfonts,amsthm,amsopn,amssymb,latexsym,soul}
\usepackage{cite}
\usepackage{color,enumitem,graphicx,cancel}
\usepackage[normalem]{ulem}
\usepackage[colorlinks=false,urlcolor=blue,
citecolor=red,linkcolor=blue,linktocpage,pdfpagelabels,
bookmarksnumbered,bookmarksopen]{hyperref}
\usepackage[english]{babel}

\usepackage[left=2.61cm,right=2.61cm,top=2.72cm,bottom=2.72cm]{geometry}
\usepackage[hyperpageref]{backref}

\usepackage[colorinlistoftodos]{todonotes}

\makeatletter
\providecommand\@dotsep{5}
\def\listtodoname{List of Todos}
\def\listoftodos{\@starttoc{tdo}\listtodoname}
\makeatother

\numberwithin{equation}{section}
\newtheorem{theorem}{Theorem}[section]
\newtheorem{lemma}[theorem]{Lemma}

\newtheorem{definition}[theorem]{Definition}
\newtheorem{proposition}[theorem]{Proposition}
\newtheorem{question}[theorem]{Question}
\newtheorem{remark}[theorem]{Remark}
\newtheorem{corollary}[theorem]{Corollary}
\newcommand{\eps}{\varepsilon}

\newcommand{\R}{\mathbb{R}}

\newcommand{\RN}{{\mathbb{R}^N}}

\newcommand{\de}{\partial}

 \DeclareMathOperator{\dv}{div}
\DeclareMathOperator{\dist}{dist} 
\DeclareMathOperator{\supp}{supp} 
\DeclareMathOperator{\dom}{dom}
\DeclareMathOperator{\range}{range}

\DeclareMathOperator{\const}{const.}
\renewcommand{\le}{\leqslant}
\renewcommand{\ge}{\geqslant}
\renewcommand{\a }{\alpha }

\renewcommand{\d }{\delta }
\newcommand{\vfi}{\varphi}

\renewcommand{\l }{\lambda}
\newcommand{\n }{\nabla }
\newcommand{\s }{\sigma }
\renewcommand{\t}{\theta}
\renewcommand{\O}{\Omega}

\newcommand{\G}{\Gamma}

\newcommand{\ci}{\mathcal{I}}
\newcommand{\cj}{\mathcal{J}}

\newcommand{\ch}{\mathcal{H}}

\newcommand{\E}{\mathcal{E}}
\newcommand{\co}{\mathcal{O}}
\newcommand{\I}{\mathcal{I}}

\newcommand{\X}{\mathcal{X}}
\newcommand{\N}{\mathbb{N}}

\newcommand{\irn }{\int_{\RN}}

\def\bbm[#1]{\mbo\X{\boldmath $#1$}}
\newcommand{\beq }{\begin{equation}}
\newcommand{\eeq }{\end{equation}}

\renewcommand{\de}{\partial}

\newcommand{\pra}{\phi_{\rho^\ast}}

\newcommand{\pran}{\phi_{\rho^{\ast,n}}^n}

\newcommand{\ran}{\rho^{\ast,n}}
\newcommand{\br}{\bar \rho}
\newcommand{\pbr}{\phi_{\bar \rho}}

\title[Equilibrium measures and equilibrium potentials in the Born-Infeld model ]{
Equilibrium measures and equilibrium potentials\\ in the Born-Infeld model}

\author[D. Bonheure]{Denis Bonheure}
\address[D. Bonheure]{\newline\indent
D\'epartement de Math\'ematique, Universit\'e libre de Bruxelles, 
\newline\indent 
CP 214, Boulevard du Triomphe, B-1050 Bruxelles, Belgium}
\email{\href{mailto:denis.bonheure@ulb.ac.be}{denis.bonheure@ulb.ac.be}}

\author[P. d'Avenia]{Pietro d'Avenia}
\author[A. Pomponio]{Alessio Pomponio}
\address[P. d'Avenia and A. Pomponio]{\newline\indent
Dipartimento di Meccanica, Matematica e Management,
Politecnico di Bari
\newline\indent
Via Orabona 4,  70125  Bari, Italy}
\email{\href{mailto:pietro.davenia@poliba.it}{pietro.davenia@poliba.it}}
\email[Corresponding Author]{\href{mailto:alessio.pomponio@poliba.it}{alessio.pomponio@poliba.it}}

\author[W. Reichel]{Wolfgang Reichel}
\address[W. Reichel]{\newline\indent
Institut fuer Analysis, Karlsruher Institut fuer Technologie 
\newline\indent
Englerstrasse 2,
76131 Karlsruhe, Germany}
\email{\href{mailto:wolfgang.reichel@kit.edu}{wolfgang.reichel@kit.edu}}


\begin{document}

\begin{abstract}
In this paper, we consider the electrostatic Born-Infeld model
\begin{equation}
\label{eqabs}
\tag{$\mathcal{BI}$}
\left\{
\begin{array}{rcll}
-\operatorname{div}\left(\displaystyle\frac{\nabla \phi}{\sqrt{1-|\nabla \phi|^2}}\right)&=& \rho & \hbox{in }\mathbb{R}^N,
\\[6mm]
\displaystyle\lim_{|x|\to \infty}\phi(x)&=& 0
\end{array}
\right.
\end{equation} 
where $\rho$ is a charge distribution on the boundary of a bounded domain $\Omega\subset \mathbb{R}^N${, with $N\ge3$}. We are interested in its {\em equilibrium measures}, i.e. charge distributions which minimize the electrostatic energy of the corresponding potential among all possible distributions  with fixed total charge.  We prove existence of  equilibrium measures and we show that the corresponding {\em equilibrium potential} is unique and constant in $\overline \Omega$. Furthermore, for smooth domains, we  obtain the uniqueness of the equilibrium measure, we give its precise expression, and we verify that the equilibrium potential solves \eqref{eqabs}. Finally we characterize balls in $\mathbb{R}^N$ as the unique sets among all bounded $C^{2,\alpha}$-domains $\Omega$ for which the equilibrium distribution is a constant multiple of the surface measure on $\partial\Omega$. The same results are obtained also for Taylor approximations of the electrostatic energy.

\medbreak 

\noindent{\sc{R\'esum\'e}.}
Dans cet article, nous consid\'erons le mod\`ele \'electrostatique de Born-Infeld
\begin{equation}
\label{eqabs}
\tag{$\mathcal{BI}$}
\left\{
\begin{array}{rcll}
-\operatorname{div}\left(\displaystyle\frac{\nabla \phi}{\sqrt{1-|\nabla \phi|^2}}\right)&=& \rho & \hbox{in }\mathbb{R}^N,
\\[6mm]
\displaystyle\lim_{|x|\to \infty}\phi(x)&=& 0
\end{array}
\right.
\end{equation} 
o\`u $\rho$  est une distribution de charges sur le bord d'un domaine born\'e $\Omega\subset \mathbb{R}^N${, $N\ge3$}. Nous nous int\'eressons aux mesures d'\'equilibre, c'est-\`a-dire aux distributions de charges qui minimisent l'\'energie \'electrostatique du potentiel correspondant, parmi toutes les distributions possibles dont la charge totale est fix\'ee.  Nous prouvons l'existence de mesures d'\'equilibre et nous montrons que le potentiel correspondant \`a l'\'equilibre est unique et constant dans $\overline \Omega$. De plus, pour les domaines r\'eguliers, nous obtenons l'unicit\'e de la mesure d'\'equilibre, nous donnons son expression pr\'ecise et nous v\'erifions que le potentiel d'\'equilibre est une solution du probl\`eme aux limites \eqref{eqabs}. Finalement, nous caract\'erisons les boules dans $\mathbb{R}^N$ comme \'etant les uniques domaines born\'es $\Omega$ de classe $C^{2,\alpha}$ pour lesquels la distribution d'\'equilibre est un multiple de la mesure de surface sur  $\partial\Omega$. Les m\^emes r\'esultats sont obtenus aussi pour des approximations de Taylor de l'\'energie \'electrostatique.
\end{abstract}
\maketitle


\begin{center}
	\begin{minipage}{11cm}
		\tableofcontents
	\end{minipage}
\end{center}

\section{Introduction and main results}

In this paper we consider the electrostatic Born-Infeld model. Let $\rho$ be a charge distribution on the boundary of a bounded domain $\Omega\subset \R^N${, with $N\ge3$} ($N=3$ being the physically relevant case), i.e., we consider $\rho$ as an element of the set $P(\partial\Omega)$ of all positive Borel measures on $\R^N$ supported on $\partial\Omega$ with $\oint_{\partial\Omega} d\rho=1$.  We consider the \emph{electrostatic potential} generated by $\rho$ as the unique minimizer $\phi_\rho$ of the Born-Infeld electrostatic action 
$$
\mathcal{I}(\phi) = \int_{\R^N}  \left(b^2-b\sqrt{b^2-|\n \phi|^2}\right) \,dx - \oint_{\partial\Omega} \phi d\rho
$$
where $\phi$ runs through the set $\mathcal{X}$ of all Lipschitz functions with Lipschitz-constant less or equal $b>0$ and $\int_{\R^N} |\nabla \phi|^2\,dx < \infty$. One of the main results of \cite{BDP} shows that $\mathcal{I}$ has a unique minimizer in $\mathcal{X}$; further uniqueness questions were addressed in \cite{K2}. For each electrostatic potential $\phi_\rho$ we can consider the Born-Infeld electrostatic energy $\mathcal{E}$ (details of the definition are given in Section~\ref{equi_measures})
$$
\mathcal{E}(\phi_\rho) := - \mathcal{I}(\phi_\rho).
$$
Among all possible charge distributions $\rho\in P(\partial\Omega)$ one can search for those distributions, which create least-energy potentials. Such a distribution $\rho^\ast\in P(\partial\Omega)$ is called \emph{equilibrium distribution} and is defined as follows:
$$
\mathcal{E}(\phi_{\rho^\ast}) = \min_{\rho\in P(\partial\Omega)} \mathcal{E}(\phi_\rho).
$$
The corresponding minimizer $\phi_{\rho^\ast}$ is called an \emph{equilibrium potential}. The main purpose of this paper is the study of the existence and the properties of $\rho^\ast$. 

\medskip

Our first set of results is as follows. We prove that 
\begin{itemize}
\item[(i)] equilibrium distributions exist, cf. Theorem~\ref{th:exis};
\item[(ii)] the equilibrium potential $\phi_{\rho^\ast}$ is unique and takes a constant value $\lambda^\ast$ in $\overline{\Omega}$, cf. Theorem~\ref{main} and Corollary~\ref{cor:unique_potential}.
\end{itemize}
An open question in the electrostatic Born-Infeld theory is the following: is it true or not that electrostatic potentials, i.e., minimizers of $\mathcal{I}$ on $\X$, are weak solutions of the electrostatic Born-Infeld equations
\begin{equation*}
\leqno{(\mathcal{BI})} \qquad \qquad\qquad \qquad
\left\{
\begin{array}{rcll}
-\operatorname{div}\left(\displaystyle\frac{ b\nabla \phi}{\sqrt{b^2-|\nabla \phi|^2}}\right)&=& \rho & \hbox{in }\mathbb{R}^N,
\\[6mm]
\displaystyle\lim_{|x|\to \infty}\phi(x)&=& 0.
\end{array}
\right.
\end{equation*}
The answer in general seems not to be known. It is however true when $\Omega$ is a ball and $\rho$ a constant multiple of the surface-measure on $\partial\Omega$, cf. \cite{BDP}. In the case where $\partial\Omega\in C^{2,\alpha}$ we can say more on the equilibrium potential. This is our second set of results. We prove that
\begin{itemize}
\item[(iii)] $\phi_{\rho^\ast}$ is a weak solution of \eqref{eq:BI} which is smooth outside $\partial\Omega$, cf. Theorem~\ref{main2};
\item[(iv)] the (unique) value $\lambda^\ast$ has a unique characterization given in Proposition~\ref{pr:rhol}: in particular 
$$
d\rho^\ast= - b\frac{\partial_\nu \phi_{\rho^\ast}}{\sqrt{b^2-|\nabla \phi_{\rho^\ast}|^2}}d\sigma,
$$
where $\partial_\nu \phi_{\rho^\ast}$ denotes the outer normal derivative of $\phi_{\rho^\ast}$ w.r.t. $\Omega$;
\item[(v)] the equilibrium distribution $\rho^\ast$ is unique, cf.~Corollary~\ref{cor:unique_measure}.
\end{itemize}
As a third set of results we consider in Section~\ref{sec:approx} a regularized Born-Infeld model, where the action is given by 
\begin{equation} \label{defi_alpha_h}
\mathcal{I}^n(\phi) = \sum_{h=1}^{n} \frac{\alpha_{h}}{2h} \|\nabla\phi\|_{2h}^{2h}  - \oint_{\partial\Omega} \phi d\rho, \qquad \alpha_1 = 1, \alpha_h=\frac{1\cdot 3 \cdot 5 \cdots (2h-3)}{2^{h-1} (h-1)!} b^{2(1-h)}. 
\end{equation}
This arises from the Taylor expansion formula $b^2-b\sqrt{b^2-x^2} = \sum_{h=1}^\infty \frac{\alpha_h}{2h}x^{2h}$ which converges for $x\in [-b,b]$. As before the corresponding electrostatic energy is given by $\mathcal{E}^n(\phi):= - \mathcal{I}^n(\phi)$.
One can define the \emph{electrostatic potential} $\phi^n_\rho$, the  \emph{equilibrium distribution} $\rho^{\ast,n}$ and the \emph{equilibrium potential} $\phi^n_{\rho^{\ast,n}}$ within the approximated model as before based on the notion of action and energy.
Our  results for this approximated model are as follows:
\begin{itemize}
\item[(vi)] there exist equilibrium distributions $\rho^{\ast,n}$ and the associated equilibrium potential $\phi^n_{\rho^{\ast,n}}$ is unique and constant on $\overline{\Omega}$; if, moreover, $\partial\Omega$ is $C^{1,\alpha}$ then $\rho^{\ast,n}$ is unique, cf. Theorem~\ref{th:exis-n}; 
\item[(vii)] the weak limit $\bar\rho$ of $\rho^{\ast,n}$ as $n\to \infty$ is an equilibrium distribution for the   Born-Infeld electrostatic model. In case $\partial\Omega\in C^{2,\alpha}$ it coincides with $\rho^\ast$, i.e., the unique equilibrium distribution of the   Born-Infeld model, cf. Proposition~\ref{prop:equilibrium-bar}.
\end{itemize}
In the final Section~\ref{sec:ball} we characterize balls in $\R^N$ in the following way:
\begin{itemize}
\item[(viii)] for the Born-Infeld model, the ball is the unique set among all bounded $C^{2,\alpha}$-domains $\Omega$ for which the equilibrium distribution $\rho^\ast$ is a constant multiple of the surface measure $d\sigma$ on $\partial\Omega$;
\item[(ix)] the same holds for the approximated model if $\rho^{\ast,n}$ is a constant multiple of $d\sigma$ on $\partial\Omega$. 
\end{itemize}
{ A regularity result for the electrostatic Born-Infeld equation can be derived from the milestone paper of Bartnik and Simon \cite{BS}. The result was extended to more general data by Bonheure and Iacopetti \cite{Bon-Iac2} but a sharp result is not known. When $\rho$ is a finite sum of point measures, we mention Kiessling's work \cite{K,K-corr}. The electrostatic Born-Infeld problem has many similarities with the mean curvature problem for graphs in Minkowski space. A variational approach to this geometric problem is given in \cite{bereanu_jebelan_mawhin}. Essential properties of the constant mean curvature problem were already considered in fundamental work by Calabi \cite{calabi}, Cheng and Yau \cite{cheng_yau}, Treibergs \cite{T82}, Bartnik and Simon \cite{BS}, Gerhardt \cite{Gerhardt}, Ecker \cite{E}, Klyachin and Miklyukov \cite{KM}, and Carley and Kiessling \cite{CK}. We refer to \cite{Lopez2014} for more insight on the geometry of Lorentz-Minkowski spaces and we also mention the recent contribution \cite{Bon-Iac1} concerning spacelike radial graphs. 

To the best of our knowledge results on the Born-Infeld electrodynamic theory are relatively sparse in the mathematical literature.  Well-posedness and stability issues are addressed in \cite{speck}.  In \cite{schroedinger} the interaction of two circularly polarized waves was shown to fulfill the Born-Infeld model in vacuum. In \cite{yang} one can find a discussion of the electro- and magnetostatic cases without sources. Special explicit solutions of the dynamical Born-Infeld model, e.g., propagating and counter-propagating plane waves, as well as connections to minimal surfaces can be found in \cite{malory}.  We also mention \cite{brenier,serre} in the context of conservation laws.

Finally we point out that the characterization of balls by their electrostatic properties in the linear Maxwell-Poisson context was achieved in \cite{mendez_reichel}, \cite{reichel} for dimensions $N\geq 2$, in \cite{ebenfelt}, \cite{gardiner} for $N=2$, and in \cite{garofalo} for the $p$-Laplacian context.}

\section{Equilibrium measures and equilibrium potentials in electrostatic theories} \label{equi_measures}

Let us consider an electromagnetic field $(\mathbf{E},\mathbf{B})$ in $\mathbb{R}^3$ described by a gauge potential $(\phi,\mathbf{A})$ as follows:
\[
\mathbf{E}=-\nabla\phi-\partial_t \mathbf{A},
\quad
\mathbf{B}=\nabla\times \mathbf{A}.
\]
In the Lagrangian formulation of electromagnetic theories like Maxwell, Born, Born-Infeld the evolution of the fields $(\mathbf{E},\mathbf{B})$ is described by the principle of stationary action, i.e., the equations of motion are the Euler-Lagrange equations of the action functional 
$$
\int_{\R^4} \mathcal{L}(\phi,\mathbf{A}, \n \phi, \n \mathbf{A},\partial_t \phi, \partial_t \mathbf{A})\,d(x,t).
$$
Here ${\mathcal L}$ is called the Lagrangian density. Let us recall the following definition of the energy of an electromagmetic field $(\mathbf{E},\mathbf{B})$  (see e.g. \cite{FOP,GF}). In the following we write $\langle\cdot,\cdot \rangle$ to denote the dual pairing between a space and its dual space.

\begin{definition}
	\label{def_energy}
Let $\mathbf{V}=(\phi,\mathbf{A})$ and suppose that $\mathcal{L}=\mathcal{L}(\mathbf{V}, \nabla\mathbf{V}, \partial_{t} \mathbf{V} )$ is the Lagrangian density of the electromagmetic field $(\mathbf{E},\mathbf{B})$. The {\em Lagrangian} at time $t$ is defined by
$$
\mathcal{I}(\mathbf{V}):=\int_{\R^3} \mathcal{L}(\mathbf{V}, \nabla \mathbf{V}, \partial_{t} \mathbf{V})\,dx
$$
and the {\em energy} at time $t$ is defined as
\[ 
\mathcal{E}(\mathbf{V})
:= 
\int_{\R^3} T^{00}(\mathbf{V}, \nabla \mathbf{V}, \partial_{t} \mathbf{V})\,dx,
\]
where
\[
T^{00}(\mathbf{V}, \nabla \mathbf{V}, \partial_{t} \mathbf{V})=\frac{\partial\mathcal{L}(\mathbf{V}, \nabla \mathbf{V}, \partial_{t} \mathbf{V})}{\partial (\partial_t \mathbf{V})}\cdot \partial_t \mathbf{V} - \mathcal{L}(\mathbf{V}, \nabla \mathbf{V},\partial_{t} \mathbf{V}).
\]
\end{definition}

In the electrostatic case $\mathbf{V}=(\phi,0)$ with $\phi$ independent of $t$, the Lagrangian becomes 
$$
\mathcal{I}(\phi) = \int_{\R^3} \mathcal{L}(\phi, \nabla \phi)\,dx
$$
and the electrostatic energy is given by
\begin{equation} \label{static_energy}
\mathcal{E}(\phi) = -\mathcal{I}(\phi)= -\int_{\R^3} \mathcal{L}(\phi,\nabla\phi)\,dx.
\end{equation}
In the classical Maxwell theory, the Lagrangian density is given by
\[
\mathcal{L}_{\rm M}
=
\frac{1}{2} (|\mathbf{E}|^2 - |\mathbf{B}|^2)  - \rho\phi + \mathbf{J}\cdot\mathbf{A},
\]
where $\rho$ is a charge distribution and $\mathbf{J}$ is a current density. According to \eqref{static_energy}  the electrostatic Maxwell energy reads
$$
{\mathcal E}_{\rm M}(\phi)=-\mathcal{I}_{\rm M}(\phi)=-\frac{1}{2}\int_{\R^3} |\nabla\phi|^2\,dx + \langle \rho, \phi \rangle.
$$
The notation $\langle \rho, \phi \rangle$ stands for the duality bracket between a finite Borel measure $\rho$ on $\R^3$ and a continuous function.
The electrostatic potential generated by a given distribution $\rho$ in absence of currents arises as the unique solution of the Euler-Lagrange equation corresponding to the Lagrangian $\mathcal{I}_{\rm M}$, i.e. the unique solution of the Poisson equation
\beq\label{eq:3}
-\Delta\phi= \rho \quad \hbox{in }\mathbb{R}^3, \quad \lim_{|x|\to\infty} \phi(x)= 0.
\eeq
For a positive Borel measure $\rho$ with compact support and with finite total mass, 
the distributional solution to \eqref{eq:3} is given by the (Newtonian) potential (see \cite[Definition 2.1, Theorem 4.1]{W})
\begin{equation}
\phi_{\rho}(x) = \frac{1}{4\pi} \int_{\supp \rho} \frac{1}{|x-y|}\,d\rho(y).
\label{newton_potential}
\end{equation}
In general, this is not a finite energy solution. For instance, it is well known that if $\rho=\delta_0$ (the Dirac-delta measure at $0$), then $\phi_{\rho}(x)= \frac{1}{4\pi |x|}$ does not have finite energy. 
Notice that,  whenever $\phi_\rho$ from \eqref{newton_potential} has finite energy, we have 
\[
\mathcal{E}_{\rm M}(\phi_\rho)  = 
\frac{1}{2}
\int_{\R^3}
|\nabla\phi_\rho|^2 \,dx
= \frac{1}{8 \pi} \int_{\supp \rho}\int_{\supp \rho} \frac{1}{|x-y|}\,d\rho(y)\,d\rho(x).
\]

Given $\Omega\subset\mathbb{R}^3$, we denote by $P(\partial\Omega)$ the set of probability measures on $\partial\Omega$, i.e., the set of positive Borel measures $\rho$ on $\R^3$ which are supported on $\partial\Omega$ and normalized by $\oint_{\partial\Omega} d\rho=1$.  For every $\rho\in P(\partial\Omega)$,   let $\phi_\rho$ be the corresponding solution of \eqref{eq:3} given by \eqref{newton_potential}. We can now define the notion of an \emph{equilibrium measure} in the Maxwell potential theory.

\begin{definition} A probability measure $\rho^\ast\in P(\partial\Omega)$ is called an equilibrium measure or equilibrium distribution for $\mathcal{E}_{\rm M}$ if
\[
\mathcal{E}_{\rm M}(\phi_{\rho^\ast})
=  \inf_{\rho \in P(\partial\Omega)} \mathcal{E}_{\rm M}(\phi_\rho) = \inf_{\rho \in P(\partial\Omega)}  \frac{1}{8 \pi} \oint_{\partial\Omega}\oint_{\partial\Omega} \frac{1}{|x-y|}\,d\rho(y)\,d\rho(x).
\]
The corresponding solution $\phi_{\rho^\ast}$ of \eqref{eq:3} given by \eqref{newton_potential}  is called equilibrium potential.
\end{definition}
Notice that the expressions for $\mathcal{E}_{\rm M}$ and the double integral may be infinite for some measures $\rho\in P(\partial\Omega)$. The infimum, however, is always  attained by a unique equilibrium measure and the equilibrium potential is constant on $\overline{\Omega}$, except for a subset of capacity zero, cf. \cite[Theorem~6.3, Theorem~7.1, and Theorem~9.1]{W}.

{ For a more recent literature about the traditional electrostatic equilibrium problem using the Coulomb potential of electrostatics we refer to \cite{L,ST}.}

\medskip

Let us now consider the Lagrangian density introduced by Born and the modified one by Born and Infeld, namely
$$
\mathcal{L}_{\rm B} = b^2  \left(1-\sqrt{1-\frac{|\mathbf{E}|^2 - |\mathbf{B}|^2}{b^2}}\right)  - \rho\phi + \mathbf{J}\cdot \mathbf{A}
$$
and 
$$
\mathcal{L}_{\rm BI}
=
b^2\left(1-\sqrt{1-\frac{|\mathbf{E}|^2 - |\mathbf{B}|^2}{b^2}-\frac{(\mathbf{E}\cdot\mathbf{B})^2}{b^4}}\right)  -\rho\phi + \mathbf{J}\cdot \mathbf{A}
$$
where $b>0$ is the Born \emph{field strength parameter} (see \cite{Bnat,B,BInat,BI}).
In the electrostatic case, Born and Born-Infeld Lagrangian densities lead  by \eqref{static_energy} to the same energy 
$$
\mathcal{E}(\phi) = -\mathcal{I}(\phi) = -b^2\int_{\R^3} \left(1 - \sqrt{1-|\nabla\phi|^2/b^2}\right)\,dx + \langle\rho,\phi\rangle.
$$
In the sequel we set $b=1$ since the exact value of $b$  has no qualitative influence on our analysis.

One can generalize the functionals $\mathcal{E, I}$ to $\R^N$ with $N\geq 3$ { and $P(\partial\Omega)$ will denote Borel probability measures on $\R^N$ supported on $\partial\Omega$ for bounded domains $\Omega\subset\R^N$}. Then,  at least formally, the Euler-Lagrange equation for the  Lagrangian $\mathcal{I}$ is given by
\begin{equation}\label{eq:BI}
\tag{$\mathcal{BI}$}
\left\{
\begin{array}{rcll}
-\operatorname{div}\left(\displaystyle\frac{\nabla \phi}{\sqrt{1-|\nabla \phi|^2}}\right)&=& \rho & \hbox{in }\mathbb{R}^N,
\\[6mm]
\displaystyle\lim_{|x|\to \infty}\phi(x)&=& 0.
\end{array}
\right.
\end{equation}
As shown in \cite{BDP} for a large class of measures, finite energy solutions of \eqref{eq:BI} can be found by minimizing the functional ${\mathcal I}$ on the set
\[
\X=D^{1,2}(\RN)\cap\{\phi\in C^{0,1}(\RN):\|\nabla\phi\|_\infty\leq 1\},
\]
which is a closed and convex subset of the space $D^{1,2}(\R^N)$.

In particular, following the definition of a weak solution \cite[Definition 1.2]{BDP}, we have that  a weak solution of \eqref{eq:BI} must minimize ${\mathcal I}$ on $\X$, cf. \cite[Proposition 2.6]{BDP}.

%

For any $\rho\in P(\partial\Omega)$, we have that $\rho \in \X^\ast$ and by \cite[Theorem~1.3]{BDP} there exists a unique $\phi_\rho\in \X$ which minimizes $\mathcal{I}$ (some details on the properties of $\mathcal{X}^\ast$ are given in Section~\ref{se:pre} after Lemma~\ref{L_infinity_estimate}). Notice that, by \cite[Lemma~2.12]{BDP}, $\phi_\rho\geq 0$.  Unfortunately, as far as we are aware, it is not known whether in general the minimizer $\phi_{\rho}$ is a solution of \eqref{eq:BI}.\footnote{{Known cases, where $\phi_\rho$ solves \eqref{eq:BI}, include $\rho\in L^\infty(\R^N)\cap\X^\ast$ and $\rho$ being a radially distributed charge density or a finite sume of point charges, cf. \cite{BDP} and \cite{K,K-corr}}.} Thus,  dealing with the Born-Infeld Lagragian, we  have to define equilibrium measures through the minimizer $\phi_{\rho}$ of $\mathcal{I}$ rather than through the solution of \eqref{eq:BI}.

\begin{definition} 
A probability measure $\rho^\ast\in P(\partial\Omega)$ is called an equilibrium measure or equilibrium distribution for $\mathcal{E}$ if 
$$
\mathcal{E}(\phi_{\rho^\ast}) =  \inf_{\rho \in P(\partial\Omega)} \mathcal{E}(\phi_\rho),
$$
where for every $\rho\in P(\partial\Omega)$ we denote by $\phi_\rho$ the unique minimizer of $\mathcal{I}$ on $\mathcal{X}$. The corresponding minimizer $\phi_{\rho^\ast}$ of $\mathcal{I}$ is called equilibrium potential.
\end{definition}

{
\begin{remark} \label{rem:rel_e_h} Another way to obtain an expression for the energy is to move from the Lagrangian (expressed in terms of $\nabla\phi$) to the Hamiltonian (expressed in terms of the dual variable $\rho$) via the Legendre transform. Formally,  this yields in the Maxwell case
\begin{align*}
\mathcal{H}_{\rm M}(\rho) &=  \sup_{\phi} \left(\langle\rho,\phi\rangle - \frac{1}{2}\int_{\R^3} |\nabla\phi|^2\,dx\right) \\
&= \langle\rho,\phi_\rho\rangle - \frac{1}{2}\int_{\R^3} |\nabla\phi_\rho|^2\,dx \\
&= \frac{1}{2}\int_{\R^3} |\nabla\phi_\rho|^2 \,dx
\end{align*}
since $\phi_\rho$ satisfies \eqref{eq:3}. Similarly, in the Born-Infeld case one finds
\begin{align*}
\mathcal{H}(\rho) &=  \sup_{\phi} \left(\langle\rho,\phi\rangle - \frac{1}{2}\int_{\R^3} 1-\sqrt{1-|\nabla\phi|^2}\,dx\right) \\
&= \langle\rho,\phi_\rho\rangle - \frac{1}{2}\int_{\R^3} 1-\sqrt{1-|\nabla\phi_\rho|^2}\,dx\\
&= -\mathcal{I}(\phi_\rho).
\end{align*}
\end{remark}
In both cases the Hamiltonian of a measure $\rho$ is the same as the energy of $\phi_\rho$. If $\phi_\rho$ satisfies \eqref{eq:BI} more can be said: as a consequence of the definition of a weak solution of \eqref{eq:BI}, see \cite[Definition 1.2]{BDP}, we have, in particular, that a weak solution $\phi_{\rho}\in \mathcal{X}$ satisfies
\begin{equation}
\label{testing_with_phi}
\int_{\R^N} \frac{|\nabla \phi_\rho|^2}{\sqrt{1-|\nabla \phi_\rho|^2}}\,dx= \langle\rho,\phi_\rho\rangle,
\end{equation}
and so, in this case, the Hamiltonian of the measure $\rho$ (and hence the energy of the solution $\phi_\rho$) is
\begin{align*}
\mathcal{E}(\phi_\rho) = \mathcal{H}(\rho)
&=
\int_{\R^N}
\left(- 1
+\sqrt{1-|\nabla\phi_\rho|^2}+
\frac{|\nabla\phi_\rho|^2}{\sqrt{1-|\nabla\phi_\rho|^2}}
\right)\,dx \\
& =
\int_{\R^N} 
\left(
\frac{1}{\sqrt{1-|\nabla\phi_\rho|^2}}
- 1
\right)\,dx =: {\mathcal K}(\phi_\rho). 
\end{align*}
In general however, we only know that a minimizer $\phi_\rho\in\mathcal{X}$ of $\mathcal{I}$ satisfies \eqref{testing_with_phi} with ``$\leq$" instead of ``$=$", and thus we only have $\mathcal{E}(\phi_\rho) =\mathcal{H}(\rho) \geq {\mathcal K}(\phi_\rho).$ 
}





\section{Preliminary results}\label{se:pre}

 Here we provide the functional analytic setting for the minimization of the Lagrangian. We prove in Lemma~\ref{continuous_dependance} that the minimizer $\phi_\rho$ depends continuously on the measure $\rho$. Thereafter we provide several tools which allow to deduce that $\phi_\rho$ attains its maximum on $\partial\Omega$, cf. Lemma~\ref{maxmax}.

\medskip

Let $\rho\in P(\partial\Omega)$. In the following we add a subscript to any functional that depends on the choice of the measure $\rho$ in its definition to emphasize this dependence and to keep the notation closer to \cite{BDP}. For instance,
$$
{\mathcal I}_\rho(\phi) := \int_{\R^N} \left(1-\sqrt{1-|\nabla\phi|^2}\right)\,dx - \langle \rho,\phi\rangle, \quad \phi \in \mathcal{X}.
$$
We also define the functional $\cj:\X\to\R$ that does not depend on the measure $\rho$ by
\[
\cj(\phi):=\int_{\R^N} \left(1-\sqrt{1-|\n \phi|^2}\right)\,dx, \quad \phi\in \mathcal{X}.
\]
Observe that, since $\|\nabla\phi\|_\infty \leq 1$, we can write ${\mathcal J}$ as the infinite series
\begin{equation}\label{serie}
\cj(\phi)=\sum_{h=1}^{+\infty} \frac{\alpha_h}{2h}\|\nabla \phi\|_{2h}^{2h}.
\end{equation}
The exact values of the coefficients $\alpha_h$ are given in \eqref{defi_alpha_h}. Next we recall some fundamental properties of the  set $\X$ (\!\!\cite[Lemma 2.1]{BDP}). 
\begin{lemma}
	\label{lemma21}
	The following assertions hold:
	\begin{enumerate}[label=(\roman*),ref=\roman*]
		\item \label{it:w1p}$\X$ is continuously embedded in $W^{1,p}(\RN)$ for all $p\ge 2^*=2N/(N-2)$;
		\item \label{it:embLinf}  $\X$ is continuously embedded in $L^\infty(\RN)$;
		\item \label{it:C0} if $\phi\in \X$, then $\lim_{|x|\to \infty} \phi(x)=0$;
		\item \label{it:wc} $\X$ is weakly closed;
		\item \label{it:compact}  if $(\phi_n)_n\subset\X$ is bounded, there exists $\bar \phi\in \X$ such that, up to a subsequence, $\phi_{n}\rightharpoonup \bar \phi$ weakly in $D^{1,2}(\R^N)$ and uniformly on compact sets.
	\end{enumerate}
\end{lemma}

 Next we give a more detailed version of \eqref{it:embLinf} from the previous lemma.

\begin{lemma}\label{L_infinity_estimate}
For any $t>\max\{N,\frac{2N}{N-2}\}$ there exists a constant $C(t)$ such that, for all $\phi\in \X$,
$$
\|\phi\|_\infty \leq C(t) \|\n \phi\|_2^\frac{2}{t}(1+\|\n \phi\|_2^\frac{2}{N}).
$$
\end{lemma}

\begin{proof} With $t$ as in the assumption, define $s = \frac{tN}{N+t}$.
Then $s<N$ and $t=s^\ast=\frac{Ns}{N-s}$. Since $t>2^*=\frac{2N}{N-2}$, we deduce that $s>2$. Recall also that, by  (\ref{it:w1p}) and (\ref{it:embLinf}) in Lemma~\ref{lemma21}, $\X$ embeds into $W^{1,t}(\R^N)$ and therefore into $L^\infty(\R^N)$. Hence, for every $\phi\in \X$, using Sobolev inequality with $t=s^\ast$ and the fact that $|\nabla\phi|\leq 1$, we get 
\[
\|\phi\|_\infty
 \leq  C(t) (\|\nabla \phi\|_t + \|\phi\|_{t})
\leq \bar C(t) (\|\nabla \phi\|_t + \|\nabla \phi\|_s) 
\leq \bar C(t) (\|\nabla\phi\|_2^\frac{2}{t}+\|\nabla \phi\|_2^\frac{2}{s})
\]
since $s,t>2$. This proves the claim because $\frac{1}{s} = \frac{1}{t}+\frac{1}{N}$.
\end{proof}

As mentioned in the introduction, one of the main results of \cite[Theorem~1.3]{BDP} states that for every $\rho\in \mathcal{X}^\ast$ there exists a unique minimizer $\phi_\rho\in \mathcal{X}$ of $\mathcal{I}_\rho$. We denote by $\X^\ast$ the dual space of $D^{1,2}(\R^N)\cap C^{0,1}(\R^N)$ endowed with the norm $\|\nabla \cdot\|_{L^2}+\|\nabla\cdot\|_{L^\infty}$. Note that $\mathcal{X}^\ast \supsetneqq D^{1,2}(\R^N)^\ast$ since, e.g., finite Borel measures on $\R^N$ are included in $\mathcal{X}^\ast$. Moreover, cf. \cite[Theorem~1.5]{BDP}, if $\rho\in L^\infty(\R^N)$ then this unique minimizer is a \emph{strictly spacelike, weak solution} in the sense given in the following two definitions.

\begin{definition}[cf. \cite{BS}]\label{def:spacelike}
Let $\Omega\subset\RN$ be open. A function $\phi\in C^{0,1}(\Omega)$ is called strictly spacelike if $\phi\in C^1(\Omega)$ and $|\n \phi(x)|< 1$ for all $x\in\Omega$.
\end{definition}

\begin{definition}[cf. \cite{BDP}]\label{def:weak_solution} Let $\rho\in {\mathcal X}^\ast$. A function $\phi\in\mathcal{X}$ is called a weak solution of \eqref{eq:BI} if  
\begin{equation} \label{eq:BI_weak}
\int_{\R^N} \frac{\nabla\phi\cdot \nabla\psi}{\sqrt{1-|\nabla\phi|^2}}\,dx = \langle \rho,\psi\rangle  \mbox{ for all } \psi\in  C_c^\infty(\R^N).
\end{equation}
\end{definition}

\begin{remark} Note that the requirement that \eqref{eq:BI_weak} holds for all $\psi \in C_c^\infty(\R^N)$ is equivalent to asking that \eqref{eq:BI_weak} holds for all $\psi \in {\mathcal X}$, which is the original definition given in \cite{BDP}. 
\end{remark}

\begin{lemma}\label{estimate_for_minimizers} There exists a constant $C>0$ such that 
$$
\|\n\phi_\rho\|_2 \leq C 
$$
for every $\rho\in P(\partial\Omega)$.
\end{lemma}

\begin{proof}
	Since $\mathcal{I}_\rho(0)=0$ and $\frac{1}{2}\tau \leq 1-\sqrt{1-\tau} $, for $\tau\in[0,1]$, we have
	\[
	0
	\geq
	 \mathcal{I}_\rho(\phi_\rho)
	\geq
	\frac{1}{2} \|\nabla \phi_\rho\|_2^2 - \|\phi_\rho\|_{L^\infty(\partial\Omega)}
	\]
	and so, by Lemma~\ref{L_infinity_estimate},
	\[
	\|\nabla \phi_\rho\|_2^2 \leq 2C(t) \|\nabla\phi_\rho\|_2^\frac{2}{t}(1+\|\nabla\phi_\rho\|_2^\frac{2}{N})
	\]
with $t>\max\{N,\frac{2N}{N-2}\}$. Thus, since $\frac{2}{t}<1$ and $\frac{2}{t}+\frac{2}{N}<1$, we get the claim. 
\end{proof}


\begin{lemma}[Continuous dependence of $\phi_\rho$ on $\rho$] 
\label{continuous_dependance}
If $\rho_k, \rho \in P(\partial \Omega)$ with $\rho_k  \rightharpoonup\rho$ as $k\to \infty$ weakly in the sense of measures, then
\begin{equation}
\label{eqlimirk}
\lim_{k\to\infty} \ci_{\rho_k} (\phi_{\rho_k}) = \ci_{\rho} (\phi_{\rho})
\end{equation}
and $\phi_{\rho_k}\to \phi_\rho$ strongly in $D^{1,2}(\R^N)$ and locally uniformly on $\R^N$ as $k\to\infty$.
\end{lemma}

\begin{proof}
The boundedness of $\{\phi_{\rho_k}\}_{k\in\N}$ (see Lemma \ref{estimate_for_minimizers}) and Lemma \ref{lemma21} imply the existence of a $\phi\in \X$ such that $\phi_{\rho_k} \rightharpoonup \phi $ weakly in $D^{1,2}(\R^N)$ and uniformly on compact sets as $k\to \infty$.
Then, since $\cj$ is convex and continuous on $\X$, it is weakly lower semicontinuous and so
\[
\cj (\phi)
\leq
\liminf_{k\to\infty} \cj (\phi_{\rho_k}).
\]
Moreover, as $k\to \infty$,
\begin{align}
|\langle\rho_k - \rho, \phi_{\rho_k}\rangle| & \leq 2 \|\phi_{\rho_k}-\phi\|_{L^\infty(\partial\Omega)}+ |\langle\rho_k - \rho, \phi \rangle| \to 0, \label{ineq_2}\\
|\langle \rho_k, \phi_{\rho_k}\rangle - \langle \rho, \phi\rangle| & \leq |\langle \rho_k - \rho, \phi\rangle| +  \|\phi_{\rho_k}-\phi\|_{L^\infty(\partial\Omega)}  \to 0. \label{ineq_3}
\end{align}
Thus,  using \eqref{ineq_2} and the fact that $\phi_{\rho_k}$ minimizes $\mathcal{I}_{\rho_k}$, we have
\begin{align*}
\ci_\rho(\phi)
&\leq
\liminf_{k\to\infty} \ci_\rho (\phi_{\rho_k})
=
\liminf_{k\to\infty} [\ci_{\rho_k} (\phi_{\rho_k}) + \langle\rho_k - \rho, \phi_{\rho_k}\rangle]
{=}
\liminf_{k\to\infty} \ci_{\rho_k} (\phi_{\rho_k}) \\
&\leq
\limsup_{k\to\infty} \ci_{\rho_k} (\phi_{\rho_k})
\le
\limsup_{k\to\infty} \ci_{\rho_k} (\phi_\rho)
=
\lim_{k\to\infty} \ci_{\rho_k} (\phi_{\rho})
=\ci_{\rho} (\phi_{\rho})
\end{align*}
that implies, by the uniqueness of the minimum, $\phi = \phi_{\rho}$ and \eqref{eqlimirk} holds. In particular,  using \eqref{ineq_3}, we have that
\begin{equation}
\label{eqllimj}
\lim_{k\to\infty} \cj (\phi_{\rho_k}) = \cj (\phi_{\rho}).
\end{equation}
On the other hand, 
by Clarkson's inequality \cite{hewitt_stromberg}, we have that
\[
\left\| \frac{\nabla \phi_{\rho_k} - \nabla \phi_{\rho}}{2} \right\|_{2h}^{2h}
+ \left\| \frac{\nabla \phi_{\rho_k} + \nabla \phi_{\rho}}{2} \right\|_{2h}^{2h}
\leq
\frac{1}{2} (\| \nabla \phi_{\rho_k} \|_{2h}^{2h} + \|  \nabla \phi_{\rho} \|_{2h}^{2h}),
\quad
h\in \N,
\]
and so
\[
\sum_{h=1}^{+\infty} \frac{\alpha_h}{2h} \left\| \frac{\nabla \phi_{\rho_k} - \nabla \phi_{\rho}}{2} \right\|_{2h}^{2h}
+ \sum_{h=1}^{+\infty} \frac{\alpha_h}{2h} \left\| \frac{\nabla \phi_{\rho_k} + \nabla \phi_{\rho}}{2} \right\|_{2h}^{2h}
\leq
\frac{1}{2} \sum_{h=1}^{+\infty} \frac{\alpha_h}{2h}  (\| \nabla \phi_{\rho_k} \|_{2h}^{2h} + \|  \nabla \phi_{\rho} \|_{2h}^{2h}),
\]
namely
\[
\cj\left(\frac{\phi_{\rho_k} - \phi_{\rho}}{2}\right)
+ \cj\left(\frac{\phi_{\rho_k} + \phi_{\rho}}{2}\right)
\leq
\frac{1}{2} (\cj(\phi_{\rho_k}) + \cj(\phi_{\rho})).
\]
By the weak lower semicontinuity of $\cj$ we see that
\[
\cj (\phi_\rho)
\leq
\liminf_{k\to\infty} \cj\left(\frac{\phi_{\rho_k} + \phi_{\rho}}{2}\right).
\]
Therefore, using \eqref{eqllimj}, we obtain
\begin{align*}
0
&\leq
\liminf_{k\to\infty} \cj\left(\frac{\phi_{\rho_k} - \phi_{\rho}}{2}\right)
\leq
\limsup_{k\to\infty} \cj\left(\frac{\phi_{\rho_k} - \phi_{\rho}}{2}\right)\\
&\leq
\limsup_{k\to\infty} \left[ \frac{1}{2} (\cj(\phi_{\rho_k} ) +\cj(\phi_{\rho})) 
- \cj\left(\frac{\phi_{\rho_k} + \phi_{\rho}}{2}\right)\right]\\
&=
\cj(\phi_{\rho}) 
- \liminf_{k\to\infty} \cj\left(\frac{\phi_{\rho_k} + \phi_{\rho}}{2}\right) 
\leq 0
\end{align*}
and thus
\begin{equation}
\label{limjdiff}
\lim_{k\to\infty} \cj\left(\frac{\phi_{\rho_k} - \phi_{\rho}}{2}\right) = 0.
\end{equation}
Hence \eqref{limjdiff} allows us to conclude since, by the inequality  $\frac{1}{2}\tau \leq 1-\sqrt{1-\tau} $, for $\tau\in[0,1]$, we have that for all $\phi\in \X$, $\|\nabla\phi\|_2^2 \leq 2 \cj(\phi)$.
\end{proof}

\begin{lemma}
\label{locate_max}
Let $\rho\in L^\infty(\R^N)$ have compact support  and let $\phi_\rho\in\mathcal{X}$ be the minimizer of $\mathcal{I}_\rho$. Suppose that $0<\max_{\R^N} \phi_\rho = \phi_\rho(x_0)$ for some point $x_0\in \R^N$. Then $x_0$ lies in $\supp\rho$ or $x_0$ lies in a bounded connected component of $\R^N\setminus\supp\rho$ where $\phi_\rho\equiv \phi_\rho(x_0)$.
In any case 
\begin{equation}
\supp\rho \cap \{x\in \R^N: \phi_\rho(x) = \max_{\R^N} \phi_\rho\} \not = \emptyset.
\label{not_empty}
\end{equation}
\end{lemma}

\begin{proof} Since $\rho\in L^\infty(\R^N)$ has compact support, we see that $\rho\in \X^\ast$. Therefore \cite[Theorem~1.5]{BDP} implies that $\phi_\rho$ is a weak, strictly spacelike solution of \eqref{eq:BI}, i.e. $\phi_\rho\in C^1(\R^N)\cap \mathcal{X}$, $|\nabla \phi_\rho|<1$ in $\R^N$ and \eqref{eq:BI_weak} holds.\\
Suppose $x_0 \not \in \supp\rho$. Then there exists an open ball $B$ centered at $x_0$ such that $\bar B\cap \supp\rho=\emptyset$ and 
\begin{equation}
\int_B a(x) \nabla\phi_\rho \cdot \nabla\psi\,dx = 0 \mbox{ for all } \psi\in C_c^\infty(B)
\label{elliptic}
\end{equation}
where $a(x)= (1-|\nabla\phi_\rho(x)|^2)^{-1/2}\geq \delta>0$ in $B$. Since $\phi_\rho$ attains its maximum on $\overline{B}$ at $x_0\in B$, the strong maximum principle (see e.g. \cite[Theorem~8.19]{GT}) applies. We then deduce that $\phi_\rho \equiv\const=\phi_\rho(x_0)$ on $\overline{B}$ and in fact on the entire connected component of $\R^N\setminus\supp\rho$ containing $x_0$. This connected component must be bounded since $\phi_\rho(x)\to 0$ as $|x|\to \infty$ by (\ref{it:C0}) in Lemma~\ref{lemma21}.
\end{proof}

To develop our theory further we next rely on the idea that one can ``mollify'' a measure as expressed in the following lemma. The proof is a straightforward consequence of the definition of the mollified measure.
\begin{lemma} \label{mollify_measure}
	Let $\mu \in P(\partial\Omega)$, $\eps>0$ and $\eta_\eps : \R^N\to [0,\infty)$ be the standard mollifier supported in $\overline{B}_\eps(0)$. Define the function $f_\eps\in L^\infty(\R^N)$ by
	$$
	f_\eps(x) := \oint_{\partial\Omega} \eta_\eps(y-x)\,d\mu_y, \quad x\in \R^N,
	$$
	and the measure $\mu_\eps \in P(\R^N)$ by 
	$$
	d \mu_\eps := f_\eps \,dx.
	$$
	Then $\mu_\eps\rightharpoonup \mu$ as $\eps \to 0$ in the sense of measures and $\supp \mu_\eps \subset \supp\mu + \overline{B}_\eps(0)$.
\end{lemma}
%
%

 We also use the following lemma on the location of maxima of locally uniformly convergent sequences of continuous functions. A proof by contradiction is again straightforward.

\begin{lemma}
\label{lemmamax} 
Let $\{\phi_k\}_{k\in\N}$ be a sequence of continuous functions  on $\R^N$  that converges locally uniformly to $\phi$ and suppose that $T:=\{x\in \R^N:\phi(x) = \max_{\R^N} \phi\}$ is compact and that for all $k\in \N$ the sets $T_k := \{x\in \R^N: \phi_k(x) = \max_{\R^N} \phi_k\}$ are contained in a fixed compact set $K$.
For any given $\delta>0$ and sufficiently large $k$ the function $\phi_k$ attains its maximum only at points in $T+\overline{B}_\delta(0)$. 
\end{lemma}

As a consequence of the previous lemmas, we can now deduce that for any $\rho \in P(\de \O)$, $\phi_\rho$   attains its maximum value on the boundary of $\Omega$.

\begin{lemma}\label{maxmax} 
For any $\rho \in P(\de \O)$, we have that $\max_{\R^n} \phi_{\rho} = \max_{\partial\Omega} \phi_{\rho}$. 
\end{lemma}
\begin{proof} Suppose that $M := \max_{\R^N} \phi_{\rho} > \max_{\partial\Omega} \phi_{\rho}$ and define the compact set $ T= \{x\in \R^N: \phi_{\rho}(x)=M\}$.
	Clearly $T\cap \partial\Omega=\emptyset$ and $\dist( T,\partial\Omega)>0$. By Lemma~\ref{mollify_measure}, let $\eps_0, \delta >0$ be so small that the mollified measure 
	$  \rho_\eps$ satisfies $\dist(\supp   \rho_\eps,  T+\overline{B}_\delta(0))>0$ for all $\eps\in (0,\eps_0]$. 
 Note that $\supp\rho_\eps\subset \partial\Omega+\overline{B}_\eps(0)\subset \overline{B}_R(0)$ for all $\eps\in (0,\eps_0]$ and a suitable $R>0$. Therefore, if a point $x$ belongs to a bounded component of $\R^N\setminus\supp\rho_\eps$ then $|x|\leq R$. By Lemma~\ref{locate_max}, applied to $\rho_\eps$ and $\phi_{\rho_\eps}$, we find that $\phi_{\rho_\eps}$ has its maximum points all within the set $\overline{B}_R(0)$ for all $\eps\in (0,\eps_0]$.   
	Since $\phi_{  \rho_\eps}\to \phi_{  \rho}$ as $\eps\to 0$ locally uniformly on $\R^N$, by Lemma~\ref{lemmamax} we know that for small $\eps>0$ the function $\phi_{\rho_\eps}$ attains its maximum over $\R^N$ only at  points $y_\eps\in T +\overline{B}_\delta(0)$. However, since $T+\overline{B}_\delta(0)$ and $\supp\rho_\eps$ are disjoint we get a contradiction to \eqref{not_empty} from Lemma~\ref{locate_max}. 
This contradiction finishes the proof.

\end{proof}

\section{Equilibrium measures and equilibrium potentials: existence and properties}\label{se:equi}

In this section, we prove our main results. We begin with existence of equilibrium measures (Theorem~\ref{th:exis}), their properties (Theorem~\ref{main}) and the uniqueness of equilibrium potentials (Corollary~\ref{cor:unique_potential}). Then, in case of $C^{2,\alpha}$-smooth domains we show that the equilibrium potential solves \eqref{eq:BI} (Theorem~\ref{main2}), we give a characterization (Proposition~\ref{pr:rhol}) and address the uniqueness questions for the equilibrium measure (Corollary~\ref{cor:unique_measure}, Proposition~\ref{pr:unici}). 

\begin{theorem}\label{th:exis}
There exists an  equilibrium measure $\rho^\ast$ for $\mathcal{E}$.
\end{theorem}

\begin{proof} 
Recall from \cite[Proposition~2.3]{BDP} that  $\ci_\rho (\phi_{\rho})=-\mathcal{E}(\phi_{\rho})<0$.  Hence we infer that
\[
\inf_{\rho \in P(\partial\Omega)} \mathcal{E} (\phi_\rho)\ge 0.
\]
Let $\{\rho_k\}_{k\in\N}$ be a minimizing sequence. Since $\partial\Omega$ is compact, the sequence $\{\rho_k\}_{k\in\N}$ is tight and hence there exists $\rho^\ast\in P(\partial\Omega)$ such that $\rho_k \rightharpoonup \rho^\ast$ (see \cite{prob}). By  Lemma~\ref{continuous_dependance} we get that 
$$\mathcal{E}(\phi_{\rho^\ast}) = \lim_{k\to\infty} \mathcal{E}(\phi_{\rho_k})= \inf_{\rho \in P(\partial\Omega)} \mathcal{E} (\phi_\rho)
$$
so that $\rho^\ast$ is an  equilibrium measure for $\mathcal{E}$.
\end{proof}

From now on,  we denote by $\rho^*$ an equilibrium measure. It satisfies the following properties.

\begin{proposition}\label{cacca}
For every $\mu \in P(\partial\Omega)$ we have the inequality
\begin{equation}
\label{ineq_equilibrium}
\lambda ^\ast := \langle  \rho ^\ast,\phi_{\rho ^\ast} \rangle \leq \langle  \mu,\phi_{\rho ^\ast} \rangle.
\end{equation}
\end{proposition}
\begin{proof}
Let $\mu$ be an arbitrary measure in $P(\partial\Omega)$.  For $t\in [0,1]$, define the probability measure $\rho_t := (1-t)\rho^{\ast}+ t \mu$.   Then
\[
\mathcal{E}(\phi_{\rho_t})-\mathcal{E}(\phi_{\rho^{\ast}})
=
\ci_{\rho^*}(\phi_{  \rho^\ast})
-\ci_{\rho^*}(\phi_{  \rho_t})
+t\langle \mu-\rho^{\ast},\phi_{\rho_t}\rangle
\leq
t\langle \mu-\rho^{\ast},\phi_{\rho_t}\rangle
\]
and
\[
\mathcal{E}(\phi_{\rho_t})-\mathcal{E}(\phi_{\rho^{\ast}})
=
-\ci_{\rho_t}(\phi_{  \rho_t})
+\ci_{\rho_t}(\phi_{  \rho^*})
+t\langle \mu-\rho^{\ast},\phi_{\rho^*}\rangle
\geq
t\langle \mu-\rho^{\ast},\phi_{\rho^*}\rangle,
\]
 for $t\in(0,1]$ imply that
\[
\langle \mu-\rho^{\ast},\phi_{\rho^{\ast}}\rangle
\leq
\frac{\mathcal{E}(\phi_{\rho_t})-\mathcal{E}(\phi_{\rho^{\ast}})}{t} 
\leq
\langle \mu-\rho^{\ast},\phi_{\rho_t}\rangle.
\]
Passing to the limit as $t\to 0^+$ and using the Lemma~\ref{continuous_dependance} we get that 
$$
\lim_{t \to 0}\frac{\mathcal{E}(\phi_{\rho_t})-\mathcal{E}(\phi_{\rho^{\ast}})}{t} \mbox{ exists and } = \langle \mu-\rho^{\ast},\phi_{\rho^{\ast}}\rangle. 
$$
Finally, since $\mathcal{E}(\phi_{\rho_t})-\mathcal{E}(\phi_{\rho^{\ast}})\geq 0$, we obtain the claimed inequality \eqref{ineq_equilibrium}.
\end{proof}

\begin{lemma} \label{ae_wrt_measure}
For every $x\in \partial \Omega$ we have $\phi_{\rho ^\ast}(x) \geq \lambda ^\ast$. Moreover $\phi_{\rho ^\ast} = \lambda ^\ast$ a.e. with respect to $\rho^\ast $ on $\partial\Omega$. 
\end{lemma}
\begin{proof} If we take $\mu=\delta_x$ for $x\in \partial\Omega$ and insert this measure into \eqref{ineq_equilibrium}, we get $\phi_{\rho ^\ast}(x)\geq \lambda ^\ast$. Moreover, by the definition of $\lambda ^\ast$ we have
$$
0 = \oint_{\partial\Omega} \underbrace{(\phi_{\rho ^\ast}(x)-\lambda ^\ast )}_{\geq 0} \ d\rho ^\ast
$$
which implies $\phi_{\rho ^\ast}=\lambda ^\ast$ a.e. with respect to the measure $\rho ^\ast$ on $\partial\Omega$. 
\end{proof}

\begin{theorem} \label{main}
For every $x\in\overline{\Omega}$ we have $\phi_{\rho ^\ast}(x)=\lambda ^\ast$. 
\end{theorem}

\begin{proof} We divide the proof in two steps: we first show that the statement is true on $\de \O$ and, then, also in $\O$.

\medskip
 
\noindent{\sc Step 1}: for every $x\in\de \Omega$ we have $\phi_{\rho^\ast}(x)=\lambda^\ast$. 
\\
Suppose $M := \max_{\partial\Omega} \phi_{\rho^\ast} > \lambda^\ast$ and define 
$$
T = \{x\in \partial \Omega : \phi_{\rho^\ast}(x)=M\}.
$$
Note that
$ \rho^\ast(T)=0$ and, since $M>  \lambda^\ast$, the continuity of $\phi_{\rho^\ast}$ implies that even $\rho^\ast\left(T +\overline{B}_\delta(0)\right)=0$ for all sufficiently small $\delta\in (0,\delta_0]$ by Lemma~\ref{ae_wrt_measure}. 
\\
Let us show that in particular  $(T+\overline{B}_{\delta_0/2}(0)) \cap \supp\rho^\ast=\emptyset$. 
Suppose for contradiction that $x\in (T+\overline{B}_{\delta_0/2}(0)) \cap \supp\rho^\ast$. Then for every open neighbourhood $N_x$ of $x$ we have $\rho^\ast(N_x)>0$, and in particular for $N_x=B_{\delta_0/2}(x)$. But $B_{\delta_0/2}(x)\subset  T+\overline{B}_{\delta_0}(0)$ and hence $\rho^\ast(N_x)=0$ so that we have reached a contradiction. 
We now have that $\dist(\supp   \rho^\ast, T+\overline{B}_{\delta_0/2}(0))>0$. 
\\
Let $\eps_0$ be so small that the mollified measure $  \rho^\ast_\eps$ (according to Lemma~\ref{mollify_measure}) satisfies $\dist(\supp   \rho^\ast_\eps, T+\overline{B}_{\delta_0/2}(0))>0$ for all $\eps\in (0,\eps_0]$. 
Arguing as in Lemma~\ref{maxmax}  we get that $\phi_{\rho_\eps^\ast}$ attains its maximum only at $T+\overline{B}_{\delta_0/2}(0)$ for sufficiently small $\eps$. 
However, since $T+\overline{B}_{\delta/2}(0)$ and $\supp\rho_\eps^\ast$ are disjoint we get, as before, a contradiction to \eqref{not_empty} from Lemma~\ref{locate_max}. This finishes the first step.

\medskip
 
\noindent{\sc Step 2}: for every $x\in \Omega$ we have $\phi_{\rho ^\ast}(x)=\lambda ^\ast$. 
\\
Let us define the function
$$
\hat\phi_{\rho^\ast}(x) = \left\{
\begin{array}{ll}
\phi_{\rho^\ast}(x), & x\in \R^N\setminus\Omega, \vspace{\jot}\\
\lambda^\ast, & x \in \Omega.
\end{array}
\right.
$$
Then $\hat\phi_{\rho^\ast} \in \X$, $\mathcal{J}(\hat\phi_{\rho^\ast}) \leq \mathcal{J}(\phi_{\rho^\ast})$ and by Step 1 
$$
\langle \rho^\ast, \hat\phi_{\rho^\ast}\rangle = \oint_{\partial\Omega} \lambda^\ast \,d\rho^\ast = \langle \rho^\ast, \phi_{  \rho^\ast}\rangle
$$ 
so that finally $\mathcal{I}_{\rho^\ast}(\hat\phi_{\rho^\ast}) \leq \mathcal{I}_{\rho^\ast}(\phi_{\rho^\ast})$. The uniqueness of the minimizer implies $\hat\phi_{\rho^\ast}=\phi_{\rho^\ast}$ and this finishes the proof.
\end{proof}

From Theorem~\ref{main} we   obtain the following result, which shows in particular the interesting fact that although in general we do not know about the uniqueness of the \emph{equilibrium measure}, the \emph{equilibrium potential} is always unique.
\begin{corollary} \label{cor:unique_potential}
For every $\mu \in P(\partial\Omega)$ we have the equality
\begin{equation}
\label{eq_equilibrium}
\lambda^\ast = \langle  \rho ^\ast,\phi_{\rho ^\ast} \rangle =\langle  \mu,\phi_{\rho ^\ast} \rangle. 
\end{equation}
Moreover, the value $\lambda^\ast$ and the equilibrium potential $\phi_{\rho^\ast}$ are unique in the following sense: whenever $\rho^\ast, \rho^{\ast\ast}$ are equilibrium measures then $\phi_{\rho^\ast}=\phi_{\rho^{\ast\ast}}$ and $\lambda^\ast=\lambda^{\ast\ast}$.
\end{corollary}

\begin{proof} The equality \eqref{eq_equilibrium} follows directly from Theorem~\ref{main}.  In particular, $\langle \rho^{\ast\ast}-\rho^\ast,\phi_{\rho^{\ast\ast}}\rangle=0$. Therefore, we have 
$$
-\mathcal{E} (\phi_{\rho^\ast})=\ci_{\rho^\ast}(\phi_{\rho^\ast})
\le  \ci_{{\rho^\ast}} (\phi_{\rho^{\ast\ast}})
=\ci_{\rho^{\ast\ast}} (\phi_{\rho^{\ast\ast}})
=-\mathcal{E} (\phi_{\rho^{\ast\ast}}).
$$
Since the two energies $\mathcal{E} (\phi_{\rho^\ast})$, $\mathcal{E} (\phi_{\rho^{\ast\ast}})$ are equal, we obtain $\ci_{\rho^\ast}(\phi_{\rho^\ast})=  \ci_{{\rho^\ast}} (\phi_{\rho^{\ast\ast}})$. By the uniqueness of the minimizer of $\ci_{\rho^\ast}$ we get $\phi_{\rho^\ast}=\phi_{\rho^{\ast\ast}}$ which also implies $\lambda^\ast=\lambda^{\ast\ast}$.
\end{proof}

\begin{lemma}\label{lambda>0}
The value $\l^*$ is strictly positive.
\end{lemma}

\begin{proof}
By Theorem \ref{main} and Lemma \ref{maxmax}  we find $\lambda^*=\max_{\de \O}\phi_{\rho^*}=\max_{\RN}\phi_{\rho^*}$.  Since moreover $\phi_{\rho^*}$ vanishes at infinity, we deduce that $\lambda^*\ge 0$. Suppose for contradiction that $\lambda^*=0$. Then, by \cite[Proposition 2.7]{BDP}, 
\[
\irn \frac{|\n \phi_{\rho^*}|^2}{\sqrt{1-|\n \phi_{\rho^*}|^2}}\ dx\le \langle \rho^*,\phi_{\rho^*}\rangle=\l^\ast=0.
\] 
Hence $\phi_{\rho^*}=0$ and so $\ci_{\rho^*}(\phi_{\rho^*})=0$, in contradiction to $\ci_{\rho^*}(\phi_{\rho^*})<0$, cf. \cite[Proposition~2.3]{BDP}.
\end{proof}

For smooth domains, we can give a more precise description of the equilibrium measure and of its properties.  

\begin{theorem}\label{main2}
Suppose that $ \de\O\in C^{2,\alpha}$, for some $\alpha>0$. Then $\phi_{\rho^*}$ is the unique weak solution of 
\begin{equation}\label{eq:BI*}
\left\{
\begin{array}{rcll}
-\operatorname{div}\left(\displaystyle\frac{\nabla \phi}{\sqrt{1-|\nabla \phi|^2}}\right)&=& \rho^* & \hbox{ in } \mathbb{R}^N,
\\[6mm]
\displaystyle\lim_{|x|\to \infty}\phi(x)&=& 0.
\end{array}
\right.
\end{equation}
Moreover, for every $r>0$ sufficiently large, $\phi_{\rho^\ast}\in C^{2,\alpha}(\overline{B}_r(0)\setminus\Omega)$ and there exists $\t=\t(r)\in(0,1)$ such that $|\n \phi_{\rho^*}|\le 1-\theta<1$ in $\overline{B}_r(0)$.
\end{theorem}

\begin{proof}
By Theorem \ref{main} we know that $\phi_{\rho^*}=\lambda^*$ in $\overline{\Omega}$. The proof will be divided into two steps. In the first step, we show local regularity and strict positivity. In the second step, boundary regularity is included.

\medskip

\noindent{\sc Step 1 (local regularity):} Let us show that $\phi_{\rho^\ast}\in C^\infty (\RN \setminus\overline{\Omega})$, that it is  strictly positive and strictly spacelike in $\RN \setminus\overline{\Omega}$ and satisfies 
\begin{equation}\label{eq:oc}
\left\{
\begin{array}{rcll}
-\operatorname{div}\left(\displaystyle\frac{\nabla \phi}{\sqrt{1-|\nabla \phi|^2}}\right)&=& 0 & \hbox{in }\mathbb{R}^N\setminus \overline{\Omega},
\\[6mm]
\phi&=&\lambda^*& \hbox{on }\de \O,
\\[2mm]
\displaystyle\lim_{|x|\to \infty}\phi(x)&=&0.
\end{array}
\right.
\end{equation}
To prove this claim we adapt some ideas of \cite{BS,BDP} to our case. Let $\co$ be an arbitrary bounded domain with smooth boundary in $\mathbb{R}^N\setminus\overline{\O}$. We define the set
\[
C_{\phi_{\rho^\ast}}^{0,1}(\co)=
\left\{\phi\in C^{0,1}(\co) \mid \phi|_{\partial \co}=\phi_{\rho^\ast}|_{\partial \co}, |\nabla \phi| \le 1\right\}
\]
and the functional $I_\co:C_{\phi_{\rho^\ast}}^{0,1}(\co)\to\mathbb{R}$ by
\[
I_\co(\phi)=\int_{\co} \Big(1-\sqrt{1-|\nabla \phi|^2}\Big) \ dx.
\]
The set of all possible {\em light rays} in $\co$  is 
\begin{equation*}
K=\left\{\overline{xy}\subset\co\mid x,y\in\partial\co, x\neq y, |\phi_{\rho^\ast}(x)-\phi_{\rho^\ast}(y)|=|x-y|\right\}.
\end{equation*}
Since $\phi_{\rho^\ast}|_{\co}$ is the unique minimizer of $I_\co$ and we infer from
 \cite[Corollary 4.2]{BS} that $\phi_{\rho^\ast}$ is strictly spacelike in $\co\setminus K$,  that it satisfies 
\[
-\operatorname{div}\left(\displaystyle\frac{\nabla \phi_{\rho^\ast}}{\sqrt{1-|\nabla \phi_{\rho^\ast}|^2}}\right)= 0 \quad \hbox{ in }\co\setminus K,
\]
and  that
\[
\phi_{\rho^\ast}(tx+(1-t)y)=t\phi_{\rho^\ast}(x)+(1-t)\phi_{\rho^\ast}(y),
\qquad
0<t<1,
\]
holds for every $x,y\in\partial\co$ such that $|\phi_{\rho^\ast}(x)-\phi_{\rho^\ast}(y)|=|x-y|$ and $\overline{xy}\subset\co$. Assume now by contradiction that $K\neq\emptyset$. Then there exist $x,y\in\partial\co$ such that $x\neq y$, 
$|\phi_{\rho^\ast}(x)-\phi_{\rho^\ast}(y)|=|x-y|$, $\overline{xy}\subset\co$. Without loss of generality we can assume that $\phi_{\rho^\ast}(x)>\phi_{\rho^\ast}(y)$. It follows that for all $t\in (0,1)$, 
\begin{equation}\label{eq:ray}
\phi_{\rho^\ast}(tx+(1-t)y)=t\phi_{\rho^\ast}(x)+(1-t)\phi_{\rho^\ast}(y)=\phi_{\rho^\ast}(y)+t|x-y|.
\end{equation}
Since, for any $R>0$ such that $\co\subset B_R$, the function ${\phi_{\rho^\ast}}|_{B_R\setminus \overline{\Omega}} $ is a minimizer of $I_{B_R \setminus \overline{\Omega}}$ on the set $C_{\phi_{\rho^\ast}}^{0,1}(B_R\setminus \overline{\Omega})$, then, by \cite[Theorem 3.2]{BS}, we have that \eqref{eq:ray} holds for all $t\in\mathbb{R}$ such that $tx +(1-t)y\in B_R\setminus \overline{\Omega}$. Therefore we can stretch the {\em light ray} until one of these two possibilities occur: either both end-points belong to $\de \O$, or, at least in one direction, the light ray is unbounded. The first case is clearly impossible since $\phi_{\rho^*}=\lambda^*$ on $\de \O$ whereas the second case contradicts the boundedness of $\phi_{\rho}$. This shows that $K=\emptyset$ and therefore the $C^\infty$-regularity of $\phi_{\rho}$ on the open set $\R^N\setminus\overline{\Omega}$ follows from \cite[Remarks at page 148]{BS}.
\\ 
Finally, suppose for contradiction that $\phi_{\rho^\ast}>0$ in $\R^N$ fails. Since $\phi_{\rho^\ast}(x)\to 0$ as $x\to \infty$ this implies that $\phi_{\rho^\ast}$ attains its non-positive minimum at a point $x_0$ in $\R^N$. By Theorem~\ref{main} and Lemma~\ref{lambda>0} we find that $x_0\not\in \overline{\Omega}$, i.e., $x_0$ belongs to a connected component $Z$ of $\R^N\setminus\overline{\Omega}$. Since $\phi_{\rho^\ast}$ is a classical solution of \eqref{eq:oc} and since the equation in \eqref{eq:oc} is locally uniformly elliptic in $\R^N\setminus\overline{\Omega}$, we find by the classical maximum-principle that $\phi_{\rho^\ast}\equiv \phi_{\rho^\ast}(x_0)\leq 0$ on $Z$. But $\partial Z\cap \partial\Omega\not = \emptyset$ and $\phi_{\rho^\ast}|_{\partial\Omega} = \lambda^\ast>0$. This contradiction shows that $\phi_{\rho^\ast}>0$ in $\R^N$.

\medskip

\noindent{\sc Step 2 (boundary regularity)}: Now we show that for every $r>0$ sufficiently large, $\phi_{\rho^\ast}\in C^{2,\alpha}(\overline{B}_r(0)\setminus\Omega)$ and there exists $\t=\t(r)\in(0,1)$ such that $|\n \phi_{\rho^*}|\le 1-\theta<1$ in $\overline{B}_r(0)$.

\medskip

\noindent Note that the claim implies that $\phi_{\rho^\ast}$ is not only the unique minimizer of $\mathcal{I}_{\rho^*}$ but actually a weak solution of \eqref{eq:BI*}. The reason is that the strictly spacelike function $\phi_{\rho^*}$ lies in the interior of $\mathcal{X}$, and hence variations in all directions can be computed for the minimizer $\phi_{\rho^\ast}$ of $\mathcal{I}_{\rho^*}$ which then turns out to be a weak solution of \eqref{eq:BI}. The uniqueness follows from \cite[Proposition 2.6]{BDP} and this ends the proof of the theorem.

\medskip

\noindent
The proof  relies on an application of \cite[Theorem 3.6]{BS}. It uses the construction of a strictly spacelike function $\bar\psi$ which has the same boundary values as $\phi_{\rho^\ast}$ on suitable sets exhausting $\R^N\setminus\Omega$. We begin with the construction of $\bar\psi$. For any $\eps>0$, let us define $ \Phi_\eps=\{x\in \RN\mid \phi_{\rho^\ast}(x)<\eps \}$, $\Sigma_\eps$ the complementary set in $\RN$ of the unbounded connected component of $ \Phi_\eps$ and $\G_\eps=\de \Sigma_\eps$. We set, moreover, 
\[
R=\max_{x\in \de \O}|x| \quad \hbox{ and }\quad R_\eps=\min_{x\in \G_\eps}|x|.
\]
We want to show that
\begin{equation}\label{bareps}
\exists \bar \eps>0 \hbox{ such that } \lambda^\ast+R<R_{\bar\eps}.
\end{equation}
Suppose by contradiction that this does not hold. Then $R_{\frac 1n}\le\l^\ast+R$ for any $n\ge 1$. 
Hence, for any $n\ge 1$, there exists $x_n\in \RN$ such that $|x_n|=R_{\frac 1n}\le  \lambda^\ast+R$  and $\phi_{\rho^\ast}(x_n)=\frac 1n$. Since $\{x_n\}_{n\in\N}\subset \bar B_{\lambda^\ast+R}$, there exists $\bar x\in \bar B_{\lambda^\ast+R}$ such that, up to a subsequence, $x_n \to \bar x$. Therefore $\phi_{\rho^\ast}(\bar x)=0$ contradicting the fact that $\phi_{\rho^\ast}>0$ in the whole $\RN$.

\medskip

\noindent
By Sard's Lemma $\Gamma_\eps$ is of class $C^\infty$ for almost all $\eps \in (0,\lambda^\ast)$. 
Since moreover $\eps \mapsto R_\eps$  is decreasing, we can find a suitable $\bar \eps>0$ such that $\G_{\bar \eps}$ is of class $C^\infty$ and  simultaneously \eqref{bareps} holds. Let us define $\psi: \R^N \to \R$, as follows
\[
\psi(x)=\left\{
\begin{array}{lll}
\lambda^\ast&& x\in  B_R,
\\[2mm]
\dfrac{\bar \eps -\l^\ast}{R_{\bar \eps}-R}(|x|-R)+\lambda^\ast && x\in B_{R_{\bar \eps}}\setminus B_R,
\\[3mm]
\bar \eps && x\in   \R^N\setminus B_{R_{\bar \eps}}.
\end{array}
\right.
\]
Observe that, in $B_{R_{\bar \eps}}\setminus B_R$, we have that $|\n \psi|=\dfrac{\lambda^\ast-\bar \eps }{R_{\bar \eps}-R}<\dfrac{\l^\ast}{R_{\bar \eps}-R}<1$, by \eqref{bareps}. 
Note that $\psi$ is a Lipschitz function on $\R^N$. By taking $R'$ slightly larger than $R$ and $R'_{\bar \eps}$ slightly smaller than $R_{\bar \eps}$ we can construct a corresponding function $\psi':\R^N\to \R$ with gradient still bounded away from $1$. Mollifying $\psi'$ we obtain a $C^\infty$-function $\bar\psi:\R^N\to \R$ such that $\bar{\psi}=\l^\ast$ on $\de \O$,  $\bar{\psi}=\bar \eps$ on $\G_{\bar \eps}$ and $|\n \bar \psi|\le 1-\t_0$ in $\R^N$ for some $\t_0>0$. The function $\bar\psi$ is therefore a strictly spacelike extension of the boundary values of $\phi_{\rho^\ast}$ to the entire set $\Sigma_{\bar\eps}\setminus\overline{\Omega}$.
We can therefore apply \cite[Theorem 3.6]{BS} to conclude that  there exists $\t_{\bar \eps}>0$ such that $|\n \phi_{\rho^\ast}|\le 1-\theta_{\bar \eps}<1$ in $\Sigma_{\bar \eps} \setminus \O$. 
The claim now follows since Sard's Lemma again implies that there exists a sequence of $\bar\eps_n \searrow 0$ such that $\Gamma_{\bar \eps_n}$ is $C^\infty$, \eqref{bareps} holds and $\Sigma_{\bar\eps_n}\setminus\Omega$ exhausts $\R^N\setminus\Omega$ as $n\to \infty$.
\end{proof}

At least for convex bounded sets $\Omega$ the function $\phi_\rho$, $\rho\in P(\partial\Omega)$, solves pointwise in $\R^N\setminus\overline \Omega$ a homogeneous Born-Infeld equation. This is achieved by adpating the proof of Theorem~\ref{main2}.

\begin{proposition}
Let $\O\subset \RN$ bounded and convex. Then, for any $\rho\in P(\de \O)$, $\phi_\rho$ is a strictly spacelike in $\RN \setminus\overline{\Omega}$ and satisfies   pointwise
\begin{equation*}
\left\{
\begin{array}{rcll}
-\operatorname{div}\left(\displaystyle\frac{\nabla \phi}{\sqrt{1-|\nabla \phi|^2}}\right)&=& 0 & \hbox{in }\mathbb{R}^N\setminus \overline{\Omega},
\\[6mm]
\phi&=&\phi_\rho & \hbox{on }\de \O,
\\[2mm]
\displaystyle\lim_{|x|\to \infty}\phi(x)&=& 0.
\end{array}
\right.
\end{equation*}
\end{proposition}

\begin{proof} Replacing the constant boundary values $\lambda^\ast$ by $\phi_\rho|_{\partial\Omega}$ one can follow the arguments of Step~1 of Theorem~\ref{main2} until one reaches the two possibilities for light rays. The second kind of light ray extending to infinity leads to a contradiction as before. However, the first kind of light ray would only extend to two points on $\partial\Omega$, where $\phi_\rho$ takes certain values (not necessarily equal to the constant $\lambda^\ast$ like for $\phi_{\rho^\ast}$). Therefore, no contradiction can be reached in general, and thus these kind of light rays must be a-priori excluded by assuming the convexity of $\Omega$.
\end{proof}

The next two propositions lead to a unique characterization of the value $\lambda^\ast$, and hence, to the uniqueness of the equilibrium measure (Corollary~\ref{cor:unique_measure}). We denote by $\nu$ the outer normal to $\partial\Omega$ and write $\de_\nu$ for the normal derivative.

\begin{proposition}\label{pr:lrho}
Let $\O\subset \RN$ be $C^{2,\alpha}$ and $\l\in (0,+\infty)$. Then there exists a unique classical solution $\phi_\l\in \X$ of 
\begin{equation}\label{eq:phil}
\left\{
\begin{array}{rcll}
-\operatorname{div}\left(\displaystyle\frac{\nabla \phi}{\sqrt{1-|\nabla \phi|^2}}\right)&=& 0 & \hbox{in }\mathbb{R}^N\setminus \overline{\Omega},
\\[6mm]
\phi&=&\l & \hbox{in }\overline{\Omega},
\\[2mm]
\displaystyle\lim_{|x|\to \infty}\phi(x)&=& 0.
\end{array}
\right.
\end{equation}
Additionally, $\phi_\lambda\equiv \lambda$ in every bounded connected component of $\R^N\setminus\overline{\Omega}$, and if $Z_0$ denotes the unbounded connected component of $\R^N\setminus\overline{\Omega}$ then $0< \phi_\lambda<\lambda$ in $Z_0$ and $\de_\nu \phi_\l<0$ on $\partial Z_0$. Moreover $\phi_\l $ is the unique solution of \eqref{eq:BI} with $\rho=\rho_\l $, where 
\begin{equation}\label{eq:rhol}
d\rho_\l:=-\frac{\de_\nu \phi_\l(x)}{\sqrt{1-|\de_\nu \phi_\l(x)|^2}}d\sigma, \qquad x\in \partial\Omega,
\end{equation}
is a positive and bounded measure supported on $\de \O$. 
\end{proposition}

\begin{remark} Note that $\rho_\lambda$ defined in \eqref{eq:rhol} is a positive and bounded measure on $\partial\Omega$ with finite total mass, but not necessarily a probability measure.
\end{remark}

\begin{proof}[Proof of Proposition \ref{pr:lrho}]
We argue as in the proof of Theorem \ref{main2}. Replacing $\lambda^\ast$ by $\lambda$ in Step~1 we obtain the existence of a unique classical solution $\phi_\l$  of \eqref{eq:phil}, which is strictly spacelike and positive in $\R^N\setminus\overline{\Omega}$. Step~2 then establishes, that $\phi_\l$ is strictly spacelike and regular up to $\partial\Omega$ and satisfies $\phi_\l\in C^{2,\alpha}(\overline{B}_r(0)\setminus\Omega)$ for sufficiently large $r>0$. In particular $\partial_\nu \phi_\lambda$ is well defined on $\partial\Omega$. Since \eqref{eq:phil} is locally uniformly elliptic we can apply the classical maximum principle  
to show that $0< \phi_\lambda<\lambda$ in the unbounded connected component $Z_0$ of $\R^N\setminus\overline{\Omega}$ and $\phi_\lambda\equiv \lambda$ in every bounded connected component of $\R^N\setminus\overline{\Omega}$. This implies  $\de_\nu \phi_\l(x)\le 0$ for all $x\in \de \O$, and the Hopf Lemma even allows to conclude $\de_\nu \phi_\l<0$ on $\partial Z_0\subset\partial\Omega$, where $\nu$ is the exterior unit normal on $\partial\Omega$.
\\
Finally,  let us show that $\phi_\l$ solves \eqref{eq:BI} with right-hand side $\rho=\rho_\l$ given by \eqref{eq:rhol}. Since $\phi_\l$ is constant in $\overline{\Omega}$, for any $\vfi\in C_c^\infty(\RN)$ we have
\begin{align*}
 \irn \frac{\nabla \phi_\l\cdot \n \vfi}{\sqrt{1-|\nabla \phi_\l|^2}}
&=
\int_{\RN \setminus \overline{\Omega}} \frac{\nabla \phi_\l \cdot \n \vfi}{\sqrt{1-|\nabla \phi_\l|^2}}
\\
&=\int_{\RN \setminus \overline{\Omega}} - \dv \left(\frac{\nabla \phi_\l}{\sqrt{1-|\nabla \phi_\l|^2}} \right)\vfi
-\int_{\de \O} \frac{\de_\nu \phi_\l}{\sqrt{1-|\de_\nu \phi_\l|^2}} \vfi \ d\s\\
&=
-\int_{\de \O} \frac{\de_\nu \phi_\l}{\sqrt{1-|\de_\nu \phi_\l|^2}} \vfi \ d\s
\end{align*}
which shows that $\phi_\lambda$ is the (unique) weak solution of \eqref{eq:BI} with $\rho_\lambda$ given by \eqref{eq:rhol}.
\end{proof}

The next proposition shows that $\l ^\ast$ is the unique value $\l $ such that  $\rho_\l $ is actually a probability measure.

\begin{proposition}\label{pr:rhol}
Let $\O\subset \RN$ be of class $C^{2,\a}$  and for all $\lambda>0$ let $\phi_\lambda$ be the unique solution of \eqref{eq:phil}. Then  the value $\l ^\ast$ from Proposition~\ref{cacca} is the unique value of $\l \in (0,+\infty)$ such that the  measure
\[
d\rho_\l=-\frac{\de_\nu \phi_\l}{\sqrt{1-|\de_\nu \phi_\l|^2}} d\sigma, \qquad x\in \partial\Omega,
\] 
is a probability measure  on $\partial\Omega$.
\end{proposition}

\begin{proof}
Step 1 in Theorem~\ref{main2} shows that $\phi_{  \rho^*}=\phi_{\lambda^*}$. Moreover, since  Theorem~\ref{main2} and Proposition~\ref{pr:lrho} also imply that $\phi_{\rho^*}$ and $\phi_{\lambda^*}$ are solutions of \eqref{eq:BI} with right-hand sides $\rho^*$ and $\rho_{\lambda^*}$ respectively, we can conclude that $\rho^*=\rho_{\lambda^*}$. Thus the existence of a value $\l$ such that $\rho_\lambda$ is a probability measure on $\partial\Omega$ is proven and we just need to show its uniqueness. We can conclude if we prove that the map
\[
\Upsilon: \l \longmapsto  \int_{\partial\Omega} d\rho_\lambda \quad \mbox{ for } \lambda\in (0,+\infty)
\]
is strictly increasing. For this purpose let $\l_1,\l_2\in (0,+\infty)$ with $\l_1<\l_2$ and, for $i=1,2$, $\phi_i:=\phi_{\l_i}$ be the corresponding solutions of \eqref{eq:phil} and   $\rho_i:=\rho_{\l_i}$. Let us observe that $\tilde \phi_1:=\phi_1+\l_2-\l_1$ and $\phi_2$ satisfy both the following problem 
\begin{equation*}
\left\{\begin{array}{rcll}
-\operatorname{div}\left(\displaystyle\frac{\nabla \phi}{\sqrt{1-|\nabla \phi|^2}}\right)&=& 0 & \hbox{in }\mathbb{R}^N\setminus \overline{\Omega},
\\[6mm]
\phi&=&\l_2 & \hbox{in }\overline{\Omega},
\end{array}
\right.
\end{equation*}
while
\[
\lim_{|x|\to \infty}\tilde \phi_1(x)= \l_2-\l_1>0=\lim_{|x|\to \infty}\phi_2(x).
\]
Next we apply a comparison argument to $\tilde \phi_1$ and $\phi_2$, cf. \cite[Theorem~10.1]{GT}. For this purpose, let $F(\xi) = (1-|\xi|^2)^{-1/2}$, $\xi\in \R^N$ with $|\xi|<1$. Notice that for $\psi:= \tilde\phi_1-\phi_2$ and $\chi_t := t\tilde\phi_1 + (1-t)\phi_2$, $t\in [0,1]$ we have 
$$
\frac{\nabla \tilde \phi_1}{\sqrt{1-|\nabla \tilde \phi_1|^2}}-\frac{\nabla \phi_2}{\sqrt{1-|\nabla \phi_2|^2}}=F(\nabla \tilde\phi_1)\nabla\tilde \phi_1- F(\nabla\phi_2)\nabla\phi_2 = \int_0^1 \frac{d}{dt}\bigl( F(\nabla \chi_t)\nabla\chi_t \bigr)\,dt = a(x)\nabla\psi,
$$
where 
$$
a(x):= \int_0^1 \left(F(\nabla\chi_t(x)) + F'(\nabla\chi_t(x))\cdot \nabla\chi_t(x)\right) dt = \int_0^1 (1-|\nabla\chi_t(x)|^2)^{-3/2}\,dt. 
$$
Arguing as in Theorem~\ref{main2} the function $a(x)>0$ is bounded in bounded subsets of $\R^N\setminus \overline{\Omega}$. Hence, $\psi$ satisfies the locally uniformly elliptic equation
\begin{equation*}
\left\{\begin{array}{rcll}
-\operatorname{div}(a(x)\nabla\psi)&=& 0 & \hbox{in }\mathbb{R}^N\setminus \overline{\Omega},
\vspace{\jot}\\
\psi&=& 0 & \hbox{in }\overline{\Omega}, \vspace{\jot}\\
\displaystyle\lim_{|x|\to \infty}\psi(x) &=& \lambda_2-\lambda_1>0.& 
\end{array}
\right.
\end{equation*}
Therefore, the maximum principle applies and states that $\psi$ attains its  zero minimum on $\partial\Omega$ and its positive maximum at infinity. On every bounded connected component $Z$ of $\R^N\setminus\overline{\Omega}$ we get that $\psi\equiv 0$ since $\psi =0$ on $\partial Z$. On the unbounded component $Z_0$ of $\R^N\setminus\overline{\Omega}$ we get $0<\psi<\lambda_2-\lambda_1$ and the Hopf Lemma implies that $\partial_\nu \psi>0$ on $\partial Z_0\subset \partial\Omega$, where $\nu$ is the exterior unit normal on $\partial\Omega$. Thus, $\partial_\nu \phi_1 =\de_\nu \tilde \phi_1> \partial_\nu \phi_2$ on $\partial Z_0$ which implies that $\Upsilon(\l_1)<\Upsilon(\l_2)$ as claimed.
\end{proof}

An immediate consequence of Proposition \ref{pr:rhol} is the following  uniqueness result.
\begin{corollary} \label{cor:unique_measure}
If $\O\subset \RN$ is of class $C^{2,\a}$ then the equilibrium measure is unique.
\end{corollary}

Without regularity assumptions on the domain we are able to show only some partial result about the uniqueness of the equilibrium measure. This result is strictly related to the conjecture that any minimizer of the  functional $\mathcal{I}_\rho$ is also a (weak) solution of the corresponding PDE.

\begin{proposition}\label{pr:unici}
Let $\Omega\subset \R^N$ be bounded with no further regularity of $\partial\Omega$. Suppose that  the equilibrium potential weakly solves \eqref{eq:BI}. Then the equilibrium measure is unique.
\end{proposition}

\begin{proof} Let $\rho^\ast$ be an equilibrium measure for which $\phi_{\rho^\ast}$ weakly solves \eqref{eq:BI} with right-hand side $\rho^\ast$. Let $\rho^{\ast\ast}$ be any other equilibrium measure. From Corollary~\ref{cor:unique_potential} we know that the equilibrium potentials $\phi_{\rho^\ast}$ and $\phi_{\rho^{\ast\ast}}$ coincide.  Since $\phi_{\rho^\ast}$ solves \eqref{eq:BI}, we have in particular
\begin{equation} \label{eq:nec_suff_for_sol}
\int_{\R^N} \frac{|\nabla\phi_{\rho^\ast}|^2}{\sqrt{1-|\nabla\phi_{\rho^\ast}|^2}}\,dx = \langle \rho^\ast,\phi_{\rho^\ast}\rangle.
\end{equation} 
But since $\langle \rho^\ast,\phi_{\rho^\ast}\rangle=\langle \rho^{\ast\ast},\phi_{\rho^{\ast\ast}}\rangle$ we get that \eqref{eq:nec_suff_for_sol} also holds for $\phi_{\rho^{\ast\ast}}$. By \cite[Remark~2.8]{BDP} this implies that $\phi_{\rho^{\ast\ast}}$ also solves \eqref{eq:BI} with right-hand side $\rho^{\ast\ast}$. Hence $\phi_{\rho^\ast}$, $\phi_{\rho^{\ast\ast}}$ are (identical) weak solutions of \eqref{eq:BI} with right-hand sides $\rho^\ast$, $\rho^{\ast\ast}$, and therefore these measures coincide.
\end{proof}

We conclude this section with the following observation: the functionals $\mathcal{E}$ and { $\mathcal{K}$ (cf. Remark~\ref{rem:rel_e_h})} coincide on solutions of \eqref{eq:BI} but they can differ at a minimizer $\phi_\rho$ of $\ci_\rho$ in case it is not a solution. However, since { ${\mathcal K}$} is well defined on each minimizer of the functional $\ci_\rho$ for all $\rho\in\X^*$ we can also study the minimization of { ${\mathcal K}(\phi_\rho)$} with respect to $\rho\in P(\partial\Omega)$. The result is given next.

\begin{proposition}
\label{prop314}
There exists a measure $\tilde \rho$ such that
\begin{equation}
\label{eqrt}
{\mathcal{K}}(\phi_{\tilde\rho}) =\inf_{\rho\in P(\partial\Omega)} {\mathcal{K}}(\phi_\rho).
\end{equation}
If $\phi_{ \tilde \rho}$ and $\phi_{  \rho^*}$ are both solutions of \eqref{eq:BI} with the respective measures, then $\tilde{\rho}=\rho^*$.
\end{proposition}
\begin{proof}
By Remark~\ref{rem:rel_e_h} we know that ${ \mathcal{K}}(\phi_\rho)\leq \mathcal{E}(\phi_\rho)$ for every $\rho\in \X^\ast$. Moreover, if we take a minimizing sequence $\{\rho_k\}_{k\in\N}\subset P(\partial\Omega)$ for {$\mathcal{K}$},
 there exists $\tilde \rho\in P(\partial\Omega)$ such that, up to a subsequence, $\rho_k \rightharpoonup \tilde \rho$ weakly in the sense of measures and, by Lemma~\ref{continuous_dependance},  $\phi_{\rho_k} \to \phi_{\tilde \rho}$ strongly in $D^{1,2}(\R^N)$.
%
%
Since {$\mathcal{K}$} is lower-semicontinuous, cf. Remark~\ref{properties_of_K}, one finds
$$
{\mathcal{K}}(\phi_{\tilde\rho}) \leq \liminf_{k} {\mathcal{K}}(\phi_{\rho_k})=\inf_{\rho\in P(\partial\Omega)} {\mathcal{K}}(\phi_\rho),
$$
and so we get \eqref{eqrt}. Finally, recalling that $\mathcal{E}$ and {$\mathcal{K}$} coincide on solutions of \eqref{eq:BI},  we observe that
\[
\inf_{\rho\in P(\partial\Omega)} \mathcal{E}(\phi_\rho)
=
\mathcal{E}(\phi_{\rho^\ast}) 
=
{\mathcal{K}}(\phi_{\rho^\ast}) 
\ge
{\mathcal{K}}(\phi_{\tilde\rho}) 
=
\mathcal{E}(\phi_{\tilde\rho}) 
\ge
\mathcal{E}(\phi_{\rho^\ast}) 
\]
 and hence equality holds. Therefore we conclude by Proposition~\ref{pr:unici}.
\end{proof}

\begin{remark} \label{properties_of_K}
The functional ${{\mathcal K}}: \X\to [0,\infty]$ is a proper, convex and lower semicontinuous functional with domain 
$$
\dom({{\mathcal K}})= \{\phi \in \X: {{\mathcal K}}(\phi)<\infty\}.
$$
Moreover, {${\mathcal K}$} is continuous on $\dom({{\mathcal K}})^\circ$. Lower-semitcontinuity follows from Fatou's lemma. The continuity on $\dom({{\mathcal K}})^\circ$ can be found in \cite[Corollary~2.5]{EkTe}. 
\end{remark}

\section{The approximated problem} \label{sec:approx}

Using \eqref{defi_alpha_h} and \eqref{serie}, we can approximate, at least formally, the equation \eqref{eq:BI} by
\begin{equation}
\label{eq:n}
\left\{\begin{array}{rcll}
-\displaystyle\sum_{h=1}^{n} \alpha_{h} \Delta_{2h} \phi&=&\rho
& 
\hbox{in } \R^N,
\\[6mm]
\displaystyle\lim_{|x|\to \infty}\phi(x)&=& 0
\end{array}
\right.
\end{equation}
where for $h\in \N$ the operator $\Delta_{2h}(\cdot) = \dv(|\nabla\cdot|^{2h-2}\nabla\cdot)$ is the $2h$-Laplacian. Weak solutions of \eqref{eq:n} can be found as critical points of the Lagrangian functional
\[
\ci^n_{\rho}(\phi)= \sum_{h=1}^{n} \frac{\alpha_{h}}{2h} \|\nabla\phi\|_{2h}^{2h} - \langle\rho,\phi\rangle_n
\]
on the space $\X_{2n}$, which is defined as the completion of $C_c^\infty(\R^N)$ with respect to the norm $\|\nabla\cdot~\!\!\|_2+~\!\!\|\nabla\cdot\|_{2n}$. A similar construction can be found in \cite{K}. Naturally, we assume that $\rho\in \X_{2n}^*$. The symbol $\langle \cdot,\cdot\rangle_n$ denotes the duality bracket between $\mathcal{X}_{2n}$ and $\mathcal{X}_{2n}^\ast$. Because of the continuity, convexity and coercivity of the functional $\mathcal{I}_\rho^n$ there exists a unique minimizer $\phi_\rho^n$  on $\mathcal{X}_{2n}$ which is also a weak solution of \eqref{eq:n} and $\ci^n_\rho(\phi_\rho^n)\leq 0$. Observe that for all $m\le n$ we have that $\X\subset \X_{2n}\subset \X_{2m}$.  
\medskip

As in the Maxwell and Born-Infeld case,  cf. Section~\ref{equi_measures}, we define the electrostatic energy as $\mathcal{E}^n(\phi):= -\mathcal{I}_\rho^n(\phi)$. For weak solutions $\phi_\rho^n$ { of \eqref{eq:n}} the energy { coincides with the expression}
\[
\mathcal{E}^{n}(\phi_\rho^n)=-\ci_\rho^n(\phi_\rho^n)= {{\mathcal K}^n(\phi_\rho^n) \quad\mbox{where}\quad {\mathcal K}^n(\phi):=}
\sum_{h=1}^{n} \frac{2h-1}{2h}\alpha_{h} \|\n \phi\|_{2h}^{2h}.
\]
For this reason,  in the discussion of equilibrium measures we will consider only the following minimization problem
\begin{equation*}
\min_{\rho\in P(\partial\O)} {\mathcal{K}}^{n}(\phi_\rho^n).
\end{equation*}

The next result shows that, for $n$ large,  the functional $\ci_\rho^n$ is well defined in $\X_{2n}$ for all $\rho\in P(\partial\O)$ and so, the existence of the unique minimizer $\phi_\rho^n$ is guaranteed.

\begin{lemma}
\label{lem41}
If $n>N/2$ then $\X_{2n}\subset C_b(\R^N)$ and hence $P(\partial\Omega)\subset \mathcal{X}_{2n}^\ast$.  Moreover, $\X_{2n}$ embeds compactly into $C_b(D)$ for every bounded set $D\subset \R^N$.
\end{lemma}
\begin{proof} 
If $\phi\in \X_{2n}$, then $|\nabla\phi|\in L^2(\R^N)$ and therefore $\phi\in L^{2^*}(\R^N)$. We will show that there exists $C>0$ such that
\[
\|\phi\|_\infty
\leq
C (\|\n \phi\|_2 + \|\n \phi\|_{2n}).
\]
We argue similarly as in \cite[Proposition 8]{FOP}. Let us consider a family of $N$-dimensional cubes $Q_k$ such that $|Q_k|=1$ and $\bigcup_k Q_k=\R^N$ and $\varphi\in C_c^\infty(\R^N)$.
For every $x,y\in\R^N$ and $t\in[0,1]$ we have
\[
\varphi(x) - \varphi(y)
=\int_{0}^{1} \nabla \varphi (tx+(1-t)y) \cdot (x-y) dt,
\]
and averaging with respect to $y$ on an arbitrary $Q_k$,
\[
\varphi(x) - \int_{Q_k} \varphi(y) dy = \int_{Q_k} \left(\int_{0}^{1} \nabla \varphi (tx+(1-t)y) \cdot (x-y) dt\right) dy.
\]
Then, if $x\in Q_k$,
	\begin{align*}
	|\varphi(x)|
	&\leq
	\left|\int_{Q_k} \varphi(y) dy\right|
	+\int_{Q_k} \left(\int_{0}^{1} |\nabla \varphi (tx+(1-t)y) \cdot (x-y)| dt\right) dy \\
	&\leq
	\|\varphi\|_{L^{2^*}(Q_k)}
	+ \sum_{i=1}^{N}\int_{Q_k} \left(\int_{0}^{1} |\partial_i \varphi (tx+(1-t)y)|  dt\right) dy \\
	&\leq
	C \|\nabla \varphi \|_2
	+ \sum_{i=1}^{N} \int_{0}^{1} \left[\frac{1}{(1-t)^N}\left(\int_{(1-t)Q_k+tx} |\partial_i \varphi (y)|^{2n}  dy\right)^{\frac{1}{2n}}(1-t)^{\frac{N(2n-1)}{2n}}\right] dt \\
	&\leq
	C \|\nabla \varphi \|_2
	+ N \int_{0}^{1} \left[(1-t)^{-\frac{N}{2n}}\|\nabla \varphi \|_{L^{2n}((1-t)Q_k+tx)}\right] dt \\
	& \leq
	C \|\nabla \varphi \|_2
	+ \frac{2Nn}{2n - N} \|\nabla \varphi \|_{2n}.
	\end{align*}
The conclusion $\X_{2n}\subset C_b(\R^N)$ follows now from a density argument. Since the $L^\infty$-estimate implies that $\X_{2n}\subset W^{1,2n}(B)$ for every ball $B\subset \R^N$ and since we have $2n>N$ we get the compact embedding of $\X_{2n}$ into $C_b(B)$.
\end{proof}

\begin{lemma}\label{lem:bddn0}
Let $n>N/2$. Then there exists a constant $C_n>0$ such that for every $\rho\in P(\partial\Omega)$
$$
\|\n \phi_\rho^n\|_2 + \|\n \phi_\rho^n\|_{2n}
\leq C_{n}. 
$$
The constant $C_n$ is uniformly bounded if $n$ is bounded away from $N/2$.
\end{lemma}

\begin{proof}
The proof of Lemma \ref{lem41} shows that there exists $C_{n}>0$ with
\[
\|\phi_\rho^n\|_\infty
\leq
C_{n} (\|\n \phi_\rho^n\|_2 + \|\n \phi_\rho^n\|_{2n}),
\]
and that $C_n$ is uniformly bounded if $n$ is bounded away from $N/2$. Then
\[
0
\geq
\ci_\rho^n(\phi_\rho^n)
\geq
\frac{\alpha_{n}}{2n} \|\n \phi_\rho^n \|_{2n}^{2n}
+\frac{\alpha_{1}}{2} \|\n \phi_\rho^n \|_{2}^{2}
- C_{n} (\|\n \phi_\rho^n\|_2 + \|\n \phi_\rho^n\|_{2n})
\]
and we conclude.
\end{proof}

Next we give the counterpart of Lemma~\ref{continuous_dependance}. The proof, based on Clarkson's inequalities, is almost the same as the one of Lemma~\ref{continuous_dependance} with Lemma~\ref{lem:bddn0} replacing Lemma~\ref{estimate_for_minimizers}. We omit  the details .

\begin{lemma}[Continuous dependence of $\phi^n_\rho$ on $\rho$] 
\label{continuous_dependence_m}
Let $n>N/2$. If $\rho_k, \rho \in P(\partial \Omega)$ with $\rho_k  \rightharpoonup\rho$ as $k\to \infty$ weakly in the sense of measures then $\phi_{\rho_k}^n\to \phi_\rho^n$ strongly in $\X_{2n}$ and locally uniformly on $\R^N$, as $k\to\infty$.
\end{lemma}

We now establish existence and uniqueness of the equilibrium measure $\rho^{\ast,n}$ for {$\mathcal{K}^n$} and show that $\phi_{\rho^{\ast,n}}^n$ is constant on $\overline{\Omega}$.   First, we assume as before only boundedness of $\Omega$ and no further regularity. 
We need to adapt the proof of Lemma~\ref{not_empty}. Since $\phi_\rho^n$ is a weak solution of \eqref{eq:n} on $\R^N$ we get that $\phi_\rho^n$ weakly solves 
$$
-\sum_{h=1}^{n} \alpha_{h} \Delta_{2h} \phi_\rho^n = 0 \quad\mbox{ on } B
$$
where $B$ is an open ball with $\overline{B}\cap \supp\rho=\emptyset$. By \cite[Lemma~1]{lieberman} we obtain that $\phi_\rho^n$ is $C^1$ on 
$\R^N\setminus\supp\rho$ and hence it satisfies \eqref{elliptic} on $B$ with $a(x) = \sum_{h=1}^n \alpha_h |\nabla \phi_\rho^n(x)|^{2h-2}$ being a 
continuous function which is locally bounded on $\R^N\setminus\supp\rho$  and bounded away from zero. Therefore, \cite[Theorem~8.19]{GT} applies and we can finish the proof as in Lemma~\ref{locate_max}. 
The remaining results of Section~\ref{se:pre} stay valid for $\phi_\rho^n$ without any change. Likewise, the results of Section~\ref{se:equi} from 
Theorem~\ref{th:exis} up to and including Lemma~\ref{lambda>0} stay true without change. This all works under the only assumption that $\Omega$ is bounded. In particular, we have that the equilibrium potential is always unique.

\medskip

It remains to establish uniqueness of the equilibrium measure, and here we need (as before) more regularity of $\partial\Omega$. We do not need any analogue for Theorem~\ref{main2} since we already have that $\phi^n_{\rho^{\ast,n}}$ is a weak solution of \eqref{eq:n} with right-hand side $\rho^{\ast,n}$. Furthermore $\phi^n_{\rho^{\ast,n}}$ weakly solves
\begin{equation}\label{eq:phil_n}
\left\{
\begin{array}{rcll}
-\sum_{h=1}^{n} \alpha_{h} \Delta_{2h}\phi&=& 0 & \hbox{in }\mathbb{R}^N\setminus \overline{\Omega},
\\[6mm]
\phi&=&\l & \hbox{in }\overline{\Omega},
\\[2mm]
\displaystyle\lim_{|x|\to \infty}\phi(x)&=& 0.
\end{array}
\right.
\end{equation}
with $\lambda=\lambda^{\ast,n}=\langle \rho^{\ast,n}, \phi^n_{\rho^{\ast,n}}\rangle$. The existence of a weak solution $\phi_\lambda$ of \eqref{eq:phil_n} for any $\lambda>0$ follows from minimization of a suitable functional in a space similar to ${\mathcal X}_{2n}$. We leave these details to the reader. The $C^{1,\alpha}$-regularity of this $\phi_\lambda$ on $\overline{B}_R(0)\setminus\Omega$ for every $R>0$ follows by assuming $\partial\Omega\in C^{1,\alpha}$ and combining Theorem~1 and Lemma~1 from \cite{lieberman}. This leads to the definition of the positive and bounded measure
\begin{equation} \label{equi_n_measure}
d\rho_\lambda^n := -\left(\sum_{h=1}^n \alpha_h |\partial_\nu \phi_\lambda^n|^{2h-2}\right)\partial_\nu \phi_\lambda^n \,d\sigma
\end{equation}
as in Proposition~\ref{pr:lrho}. Note that due to the uniform ellipticity of \eqref{eq:phil_n} in a neighbourhood of the boundary, the Hopf Lemma shows that $d\rho_\lambda^n$ is strictly positive on the boundary of the unbounded connected component of $\R^N\setminus\overline{\Omega}$. The characterization of the value $\lambda^{\ast,n}$ as the unique value such that $d\rho_\lambda^n$ is a probability measure on $\partial\Omega$ is then established in the same way as in Proposition~\ref{pr:rhol} with the direct consequence of uniqueness of the equilibrium measure $\rho^{\ast,n}$. Notice that this works under the assumption $\partial\Omega\in C^{1,\alpha}$. The following statement summarizes these results.

\begin{theorem}\label{th:exis-n} 
Let $n>N/2$ and assume $\partial\Omega\in C^{1,\alpha}$. There exists a unique equilibrium measure $\rho^{\ast,n}\in P(\partial\Omega)$ for {$\mathcal{K}^n$}. Moreover, $\phi_{\rho^{\ast,n}}^n=\lambda^{\ast,n}$ on $\bar{\Omega}$ and for every $\mu \in P(\partial\Omega)$
\begin{equation}
\label{ineq_equilibrium_approx}
\lambda^{\ast,n} := \langle \rho^{\ast,n},\phi^n_{\rho^{\ast,n}} \rangle = \langle  \mu,\phi^n_{\rho^{\ast,n}} \rangle.
\end{equation}
\end{theorem}

\begin{remark} \label{c2-alpha}
Based on the local $C^{1,\alpha}$ regularity of $\phi^n_{\rho^{\ast,n}}$ we can also improve the regularity.  In case $\partial\Omega\in C^{2,\alpha}$ we can interpret the differential equation $-\sum_{h=1}^n \alpha_h \Delta_h \phi^n_{\rho^{\ast,n}}=0$ on $\R^N\setminus\Omega$ as a linear differential equation for $\phi^n_{\rho^{\ast,n}}$ with locally $\alpha$-H\"older continuous and locally uniformly elliptic coefficients. This leads to local $C^{2,\alpha}$ regularity up to the boundary and in particular that $\phi^n_{\rho^{\ast,n}}$ is a classical solution on $\R^N\setminus\Omega$. 
\end{remark}

The remaining part of this section is devoted to show that the sequence $\rho^{\ast,n}$ weakly converges to a limit measure $\bar \rho$ (Lemma~\ref{convergence_by_approximation}) which is in fact an equilibrium measure for the Born-Infeld energy $\E$ (Proposition~\ref{prop:equilibrium-bar}).

\begin{lemma}\label{le:monotonicity}
Let $N/2<m<n$ and $\rho\in P(\partial\Omega)$. Then
\begin{equation*}
\ci_\rho^m(\phi_\rho^m)< \ci_\rho^n(\phi_\rho^n)
\end{equation*}
and
\begin{equation*}
{\mathcal{K}}^n(\phi_\rho^n)< {\mathcal{K}}^m(\phi_\rho^m).
\end{equation*}
\end{lemma}

\begin{proof}The conclusion follows easily observing that
\[
\ci_\rho^m(\phi_\rho^m)\leq \ci_\rho^m(\phi_\rho^n)< \ci_\rho^n(\phi_\rho^n)
\]
and, since $\phi_\rho^n,\phi_\rho^m$ are weak solutions,
\[
{\mathcal{K}}^n(\phi_\rho^n)
= - \ci_\rho^n(\phi_\rho^n) 
< - \ci_\rho^m(\phi_\rho^m)
= {\mathcal{K}}^m(\phi_\rho^m).
\]
\end{proof}

For $n\in \N$, let us consider ${\mathfrak{k}}^n:=\min_{\rho\in P(\de \O)}{\mathcal{K}}^n(\phi_{  \rho}) ={\mathcal{K}}^n(\phi^n_{\rho^{\ast,n}})$ and $\ci^n := \min_{\psi\in \X_{2n}} \ci^n_{\rho^{\ast,n}}(\psi) = \ci^n_{\rho^{\ast,n}}(\phi^n_{\rho^{\ast,n}})$. As an immediate consequence of Lemma~\ref{le:monotonicity} and of the fact that $\ci^n = - {\mathfrak{k}}^n$, we have the following monotonicity result.
\begin{lemma}\label{le:hndec} The sequence ${\{\mathfrak{k}}^n\}_{n>N/2}$ is strictly decreasing and the sequence $\{\ci^n\}_{n>N/2}$ is strictly increasing.
\end{lemma}


\begin{lemma}\label{convergence_by_approximation}
There exists $\bar \rho\in P(\de \O)$ such that $\rho^{\ast,n} \rightharpoonup \bar \rho$ weakly in the sense of measures, $\phi^n_{\rho^{\ast,n}} \rightharpoonup \phi_{\bar \rho}$ weakly in $\X_{2m}$ for all $m> N/2$ and uniformly on compact subsets of $\R^N$ as $n\to \infty$. Here $\phi_{\bar \rho}\in \mathcal{X}$ is the unique minimizer of $\ci_{\bar \rho}:\mathcal{X} \to \R$. Moreover 
\begin{equation}\label{barrho<}
{{\mathcal K}}^n(\phi_{\bar \rho})\le \lim_n {{\mathcal K}}^n (\pran).
\end{equation}
\end{lemma}

\begin{proof} Let $\bar \rho\in P(\partial \Omega)$  such that, up to subsequences, 
$\rho^{\ast,n}\rightharpoonup \bar \rho$ weakly in the sense of measures, as $n \to \infty$. Let $m,n\in \N$ with $N/2< m\le n$. Since 
\[
{{\mathcal K}}^m(\phi^n_{\rho^{\ast,n}})\le {{\mathcal K}}(\phi^n_{\rho^{\ast,n}})={\mathfrak{k}^n\le \mathfrak{k}^1}, 
\]
we see that $\|\n \phi^n_{\rho^{\ast,n}}\|_{2}+\|\n \phi^n_{\rho^{\ast,n}}\|_{2m}\le C(m)$.  Therefore, up to a subsequence, $\{\phi^n_{\rho^{\ast,n}}\}_{n\in\N}$ is weakly convergent in $\X_{2m}$ and, by Lemma~\ref{lem41}, locally uniformly in $\R^N$. By a  diagonalization argument, we deduce that the weak limit $\phi\in \bigcap_{m\ge 1}\X_{2m}$.
The fact that $\phi \in \X$ (i.e., $\|\nabla\phi\|_\infty \leq 1$) follows as in \cite[Theorem 5.2]{BDP}. Let $\ci:= \lim_{n\to \infty} \ci^n$. 

Let us show that $\ci=\ci_{\bar \rho}(\phi)$. First note that
$$
\ci^n \leq \ci^n_{\rho^{\ast,n}}(\phi) \leq \ci_{\rho^{\ast,n}}(\phi) = \ci_{\bar \rho}(\phi)+\underbrace{\langle \bar \rho-\rho^{\ast,n},\phi\rangle}_{\to 0 \mbox{ as } n\to\infty}
$$
which implies that $\ci \leq \ci_{\bar \rho}(\phi)$. Now let us show the reverse inequality. Next note that since $\ci^m_{\rho^{\ast,m}}(\phi^n_{\rho^{\ast,n}}) \leq \ci^n_{\rho^{\ast,m}}(\phi^n_{\rho^{\ast,n}})$, for all $m\le n$, and by weak lower semicontinuity
\begin{align*}
\ci^m_{\rho^{\ast,m}}(\phi) 
& \leq \liminf_n \ci^m_{\rho^{\ast,m}}(\phi^n_{\rho^{\ast,n}}) \leq \liminf_n \underbrace{\ci^n_{\rho^{\ast,n}}(\phi^n_{\rho^{\ast,n}})}_{=\mathcal{I}^n} + \lim_{n\to \infty} \langle \rho^{\ast,n}-\rho^{\ast,m}, \phi^{n}_{\rho^{\ast,n}}\rangle \\
& = \ci + \langle \bar \rho-\rho^{\ast,m},\phi\rangle
\end{align*}
where we have used that $\phi^{n}_{\rho^{\ast,n}}\to \phi$ uniformly on $\partial\Omega$ as $n\to \infty$. Subtracting $\langle \bar \rho- \rho^{\ast,m},\phi\rangle$ from both sides, we get
$$
\ci^m_{\bar \rho}(\phi) \leq \ci.
$$
Passing to the limit on $m$ we get $\ci_{\bar \rho}(\phi) \leq \ci$. Altogether we have obtained the claim $\I_{\bar\rho}(\phi)=\I$.

\medskip

To see that $\phi=\phi_{\bar \rho}$ observe first that we have the inequalities
$$
\ci^n \leq \ci^n_{\rho^{\ast,n}}(\psi) \leq \ci_{\bar\rho}(\psi) + \langle \bar \rho-\rho^{\ast,n},\psi\rangle \quad  \mbox{ for all } \psi\in \X.
$$
This implies
$$\ci_{\bar \rho}(\phi)= \ci = \lim_{n\to \infty} \ci^n \leq \ci_{\bar \rho}(\psi) \quad \mbox{ for all } \psi\in \X$$
so that $\phi=\phi_{\bar \rho}$. 

\medskip

Finally, it remains to show that \eqref{barrho<} holds. By the monotonicity of {$\mathfrak{k}^n$} from Lemma \ref{le:hndec}  and {$\mathfrak{k}^n\geq 0$} we infer that ${\mathfrak{k}}^\infty:= \lim_{n\to \infty} {\mathfrak{k}}^n$ exists. Notice that the lower semicontinuity and the fact that ${\mathcal{K}}^m(\phi^n_{\rho^{\ast,n}}) \le {\mathcal{K}}^n(\phi^n_{\rho^{\ast,n}})$ for $m\le n$, imply that, for every $m\in\N$, we have 
\[
{\mathcal{K}}^m(\phi_{\bar \rho}) \leq \liminf_{n} {\mathcal{K}}^m(\phi^n_{\rho^{\ast,n}}) 
 \leq \liminf_{n} {\mathcal{K}}^n(\phi^n_{\rho^{\ast,n}}) 
 = {\mathfrak{k}}^\infty.
\]
Hence \eqref{barrho<} holds. 
\end{proof}

Now we can finally show that the weak limit $\bar\rho$ of the sequence $\rho^{\ast,n}$ is indeed an equilibrium measure for the Born-Infeld energy $\E$.

\begin{proposition} 
\label{prop:equilibrium-bar}
The measure $\bar\rho$ is an equilibrium measure for $\mathcal{E}$ considered in the Section~\ref{se:equi}. In particular
we have $\phi_{  \rho^\ast}=\phi_{\bar \rho}$.  If $\partial\Omega$ is $C^{2,\alpha}$ for some $\alpha>0$ then $\bar\rho$ is the unique equilibrium measure for $\mathcal{E}$.
\end{proposition}

\begin{proof}
By \eqref{ineq_equilibrium_approx} and Lemma \ref{convergence_by_approximation}, we have that for every $\mu \in P(\partial\Omega)$
\begin{equation*}
\bar \lambda:= \langle \bar \rho,\phi_{\bar \rho} \rangle = \langle  \mu,\phi_{\bar \rho} \rangle.
\end{equation*}
Hence
\[
\ci_\mu(\phi_\mu)\le \ci_\mu(\phi_{\bar \rho})=\ci_{\bar \rho}(\phi_{\bar \rho})+\langle \bar \rho-\mu ,\phi_{\bar \rho}\rangle  = \ci_{\bar \rho}(\phi_{\bar \rho}),
\]
which implies that ${\bar \rho}$ is an equilibrium measure for $\mathcal{E}$. Since by Corollary~\ref{cor:unique_potential} we know that the equilibrium potential for ${\mathcal E}$ is unique, this implies $\phi_{  \rho^\ast}=\phi_{\bar \rho}$. Uniqueness of the equilibrium measure under the assumption $\partial\Omega\in C^{2,\alpha}$ is already contained in Corollary~\ref{cor:unique_measure}. 
\end{proof}

\section{Characterization of balls via equilibrium measures} \label{sec:ball}

\begin{theorem} \label{ball}
Let $\Omega\subset \R^N$ be of class $C^{2,\alpha}$ and let $\phi_{\rho^\ast}$ be the  equilibrium potential for the   Born-Infeld
electrostatic model. Then $\Omega$ is a ball if and only if the  equilibrium distribution $\rho^\ast$ is a constant multiple of the surface
measure $d\sigma$ on $\partial\Omega$. The same characterization is true for the approximated model. 
\end{theorem}

\begin{proof}
One direction of the theorem is trivial: if $\Omega$ is a ball, then the uniqueness of the equilibrium potentials (see Corollary~\ref{cor:unique_potential} and Theorem~\ref{th:exis-n}) implies their radial symmetry. In particular, $\partial_\nu \phi_{\rho^\ast}$ and $\partial_\nu \phi^n_{\rho^{\ast,n}}$ are constant on the boundary of the ball. Therefore, using Proposition~\ref{pr:rhol} for $\phi_{\rho^\ast}$ and its counterpart, formula \eqref{equi_n_measure}, for $\phi^n_{\rho^{\ast,n}}$, the corresponding equilibrium measures given by 
\begin{equation} \label{equi_measure_ball}
d\rho^\ast=-\frac{\de_\nu \phi_{\rho^\ast}}{\sqrt{1-|\de_\nu \phi_{\rho^\ast}|^2}} d\sigma
\end{equation}
and 
\begin{equation} \label{equi_n_measure_ball}
d\rho^{\ast,n} = -\left(\sum_{h=1}^n \alpha_h |\partial_\nu \phi^n_{\rho^{\ast,n}}|^{2h-2}\right)\partial_\nu \phi^n_{\rho^{\ast,n}} \,d\sigma
\end{equation}
are both are constant multiples of $d\sigma$ on the boundary of the ball.

\medskip

\noindent Now we consider the nontrivial part of the theorem. It will follow from an application and slight modification of \cite[Theorem~1]{reichel}, which itself is based on the moving plane method of Alexandrov-Serrin \cite{Alexandrov62, Serrin71}. Define 
$$
g(s) = \frac{1}{\sqrt{1-s^2}} \mbox{ for } s\in [0,1) \mbox{ and } g_n(s) = \sum_{h=1}^n \alpha_h s^{2h-2} \mbox{ for } s\in [0,\infty)
$$
with $\alpha_h$ as given in \eqref{defi_alpha_h}. Notice that 
$$
g \in C^2[0,1), \mbox{ and } g(s)>0, (g(s)s)'>0 \mbox{ for all } s\in [0,1)
$$
and
$$
g_n \in C^2[0,\infty), \mbox{ and } g_n(s)>0, (g_n(s)s)'>0 \mbox{ for all } s\in [0,\infty).
$$
 Observe that the function $g_n$ directly satisfies condition (I)$_1$ of \cite[Theorem~1]{reichel}, whereas $g$ satisfies (I)$_1$ not on $[0,\infty)$ but only on $[0,1)$. As we will see later this is still sufficient since $[0,1)$ covers the range of $|\nabla \phi_{\rho^\ast}|$.

Next we recall the properties of $\phi_{\rho^\ast}$ and $\phi^n_{\rho^\ast,n}$. Since we are assuming that $\rho^*$ and $\rho^{*,n}$ are constant multiples of the surface measure $d\sigma$ on $\partial\Omega$, and by \eqref{equi_measure_ball} and \eqref{equi_n_measure_ball}, we get that the functions $g(\partial_\nu \phi_{\rho^\ast}) \partial_\nu \phi_{\rho^\ast}$ and $g_n(\partial_\nu \phi^n_{\rho^{\ast,n}}) \partial_\nu \phi^n_{\rho^{\ast,n}}$ are both constant on $\partial\Omega$. Since $g(s)s$ and $g_n(s)s$ are strictly increasing this implies that $\partial_\nu \phi_{\rho^\ast}$ and 
$\partial_\nu \phi^n_{\rho^{\ast,n}}$ are both constant on $\partial\Omega$. Proposition~\ref{pr:lrho} implies for $\phi_{\rho^\ast}$ that this constant is strictly negative. The same is true for $\phi^n_{\rho^{\ast,n}}$ by the counterpart of Proposition~\ref{pr:lrho} explained after \eqref{equi_n_measure}. The fact that $\partial_\nu \phi_{\rho^\ast}$ is strictly negative on $\partial\Omega$ implies again by Proposition~\ref{pr:lrho} that $\R^N\setminus\overline{\Omega}$ only consists of one connected (unbounded) component and $0<\phi_{\rho^*}< \l^\ast$ in $\R^N\setminus\overline{\Omega}$. The same holds for $\phi^n_{\rho^{\ast,n}}$.

We can now list the properties of the equilibrium potentials. Using Theorem \ref{main2}, Proposition~\ref{pr:lrho} and Proposition~\ref{pr:rhol},  for every ball $B_r(0)$, with $r>0$ large enough, the potential $\phi_{\rho^\ast}$ is a $C^{2,\alpha}(\overline{B_r(0)}\setminus\Omega)$-solution of 
\begin{equation}\label{eqi_bi}
\left\{
\begin{array}{rcll}
-\operatorname{div}(g(|\nabla \phi_{\rho^{\ast}}|)\nabla\phi_{\rho^\ast})&=& 0,  &\hbox{in }\mathbb{R}^N\setminus \overline{\O},
\\
 0\ < \ \phi_{\rho^\ast} &<& \lambda^\ast, &
\hbox{in }\mathbb{R}^N\setminus \overline{\O},
\\
\phi_{\rho^\ast}&=&\lambda^\ast & \hbox{on }\partial \O,
\\
\partial_\nu \phi_{\rho^\ast}&=&\const <0 & \hbox{on }\partial \O,
\\
\displaystyle\lim_{|x|\to \infty}\phi_{\rho^\ast}(x)&=& 0
\end{array}
\right.
\end{equation}
with $\sup_{\overline{B_r(0)}} |\nabla \phi_{\rho^\ast}|= 1-\theta(r)<1$ for some $\theta(r)\in (0,1)$. 

Similarly, as stated in Remark~\ref{c2-alpha}, for every ball $B_r(0)$, with $r>0$ large enough, the potential $\phi^n_{\rho^{\ast,n}}$ is a $C^{2,\alpha}(\overline{B_r(0)}\setminus\Omega)$-solution of  
\begin{equation}\label{eqi_bi_approx}
\left\{
\begin{array}{rcll}
-\operatorname{div}( g_n(|\nabla \phi^n_{\rho^{\ast,n}}|)\nabla\phi^n_{\rho^{\ast,n}})&=& 0,   & \hbox{in }\mathbb{R}^N\setminus \overline{\O},
\\
0\ <\ \phi^n_{\rho^{\ast,n}} &<& \lambda^\ast & \hbox{in }\mathbb{R}^N\setminus \overline{\O},
\\
\phi^n_{\rho^{\ast,n}}&=& \lambda^{\ast,n}& \hbox{on }\partial \O,
\\
\partial_\nu \phi^n_{\rho^{\ast,n}}&=&\const < 0 & \hbox{on }\partial \O,
\\
\displaystyle\lim_{|x|\to \infty}\phi^n_{\rho^{\ast,n}}(x)&=& 0.
\end{array}
\right.
\end{equation}
To finish the proof, let us first consider the approximated model  \eqref{eqi_bi_approx}. Here the function $g_n$ and the problem \eqref{eqi_bi_approx} directly fulfill the assumptions of case (I) of  \cite[Theorem~1]{reichel} which yields radial symmetry of $\phi^n_{\rho^{\ast,n}}$ and that $\Omega$ is a ball. 

Now let us consider the Born-Infeld model \eqref{eqi_bi}. Here the function $g$ is not defined on $[0,\infty)$ as required by \cite[Theorem~1]{reichel}, but only on $[0,1)$, where it also satisfies the requirement (I)$_1$, namely $g\in C^2[0,1)$, $g(s)>0, (g(s)s)'>0$ on $[0,1)$. If one investigates the proof of \cite[Theorem~1]{reichel} then it is clear that (I)$_1$ is required for the linearized equation fulfilled by $w(x) = \phi_{\rho^*} (x^\lambda)-\phi_{\rho^*}(x)$. Here $x^\lambda$ stands for the reflection of $x$ at a hyperplane $T_\lambda:=\left\{x\in\R^N:x_1=\l \right\}$ and $\lambda$ denotes the geometrically critical position of the hyperplane. This linearized equation is determined as 
\begin{equation} \label{BI_linearized}
\operatorname{div}(A(x)\nabla w)=0
\qquad\hbox{ on } \{x\in \R^N: x_1>\lambda \hbox{ and } x^\lambda\not\in \overline{\Omega} \},
\end{equation}
where
\[
A(x) \nabla w
:=
\int_0^1 \frac{d}{dt} [g(|\nabla w_t |) \nabla w_t] dt
=\left[\int_0^1 \left(\frac{1}{\sqrt{1-|\nabla w_t|^2}}\operatorname{Id}+ \frac{\nabla w_t \otimes \nabla w_t}{(1- |\nabla w_t|^2)^{3/2}}\right) dt \right] \nabla w
\]
and $ w_t(x) := t\phi_{  \rho^*}(x^\lambda)+(1-t)\phi_{  \rho^*}(x)$. Clearly, when $x$ ranges in a compact subset $K$ of the domain of definition of $w$ then $|\nabla w_t(x)|$ takes values in a compact subset of $[0,1)$ and $A(x)$ is uniformly positive definite on $K$. As a consequence, \eqref{BI_linearized} is uniformly elliptic in $K$. This is sufficient for the proof of \cite[Theorem~1]{reichel} when applied to problem \eqref{eqi_bi}. In fact, this shows that it is enough to satify the requirement (I)$_1$ on $\range(|\nabla \phi_{\rho^\ast}|)$ instead of $[0,\infty)$. As a result, we get that  $\phi_{\rho^\ast}$ is radially symmetric and that $\Omega$ is a ball. This completes the proof of Theorem~\ref{ball}.
\end{proof}

\section*{Acknowledgements}

D. Bonheure is supported by MIS F.4508.14 (FNRS), PDR T.1110.14F (FNRS) \& ARC AUWB-2012-12/17-ULB1- IAPAS. 
P. d'Avenia and A. Pomponio are partially supported by the group GNAMPA of INdAM, by FRA2016 of Politecnico di Bari and  by PRIN 2017JPCAPN {\em Qualitative and quantitative aspects of nonlinear PDEs}. W. Reichel gratefully acknowledges funding by the Deutsche Forschungsgemeinschaft (DFG, German Research Foundation) – Project-ID 258734477 – SFB 1173.

\end{document}